\documentclass[letterpaper,11pt]{amsart}

\usepackage{graphicx,color,epsfig}
\usepackage{amsfonts,amsmath,amssymb,amsthm,enumerate,mathtools}
\usepackage{comment}
\usepackage{hyperref}
\usepackage[normalem]{ulem}

\newtheorem{thm}{Theorem}[section]
\newtheorem{prop}[thm]{Proposition}
\newtheorem{lem}[thm]{Lemma}
\newtheorem{cor}[thm]{Corollary}

\theoremstyle{remark}
\newtheorem{rem}[thm]{Remark}

\theoremstyle{definition}

\newcommand{\C}{\mathbb C}     
\newcommand{\N}{\mathbb N}     
\newcommand{\R}{\mathbb R}     
\newcommand{\Z}{\mathbb Z}     

\renewcommand{\a}{\alpha}

\renewcommand{\d}{\delta}

\newcommand{\e}{\varepsilon}

\renewcommand{\l}{\lambda}

\newcommand{\s}{\sigma}
\renewcommand{\k}{\kappa}
\newcommand{\vphi}{\varphi}

\newcommand{\bP}{\overline{\mathbb{P}}}
\newcommand{\bE}{\overline{\mathbb{E}}}

\newcommand{\bigo}{\mathcal{O}}

\newcommand{\fl}[1]{\lfloor #1 \rfloor}  
\newcommand{\ind}[1]{ \mathbf{1}_{ \{ #1 \} } } 

\DeclareMathOperator{\Var}{Var}
\DeclareMathOperator{\Cov}{Cov}

\newcommand{\w}{\omega}              
\renewcommand{\P}{\mathbb{P}}        
\newcommand{\E}{\mathbb{E}}          
\newcommand{\vp}{\mathrm{v}}       
\newcommand{\bvp}{\mathrm{\mathbf{v}}} 

\newcommand{\dd}{\mathrm{d}}
\providecommand{\Abs}[1]{\Bigr\lvert#1\Bigl\rvert}
\newcommand{\nn}{\nonumber}

\newcommand{\avar}{V}
\newcommand{\bvar}{W}
\newcommand{\vvar}{\mathbf{Z}}
\newcommand{\asum}{X}
\newcommand{\bsum}{Y}
\newcommand{\vsum}{\mathbf{S}}
\newcommand{\polya}{J}
\begin{document}

\title[Regenerative CLT rates]{Berry-Esseen estimates for regenerative processes under weak moment assumptions}

\author{Xiaoqin Guo}
\address{Xiaoqin Guo \\ University of Wisconsin, Madison \\ Department of Mathematics \\ 425 Van Vleck Hall \\ Madison, WI 53706 \\ USA }
\email{guoxq84@gmail.com}
\urladdr{https://sites.google.com/site/guoxx097/}

\author{Jonathon Peterson}
\address{Jonathon Peterson \\  Purdue University \\ Department of Mathematics \\ 150 N University St \\ West Lafayette, IN 47907 \\ USA}
\email{peterson@purdue.edu}
\urladdr{http://www.math.purdue.edu/~peterson}
\thanks{J. Peterson was partially supported by NSA grants H98230-15-1-0049 and H98230-16-1-0318.}


\date{\today}

\begin{abstract}
We prove Berry-Esseen type rates of convergence for central limit theorems (CLTs) of regenerative processes which generalize previous results of Bolthausen under weaker moment assumptions. 
We then show how this general result can be applied to obtain rates of convergence for 
(1) CLTs for additive functionals of positive recurrent Markov chains under certain conditions on the strong mixing coefficients, and 
(2) annealed CLTs for certain ballistic random walks in random environments. 
\end{abstract}

\maketitle
  
\section{Introduction}

A real-valued stochastic process $\{X_n\}_{n\ge 0}$ is called a (discrete time) {\it regenerative process} if there exists an increasing sequence of random times (not necessarily stopping times) $0=\tau_0<\tau_1<\tau_2<\tau_3<\cdots$ such that if $\mathcal{G}_m  = \s(\tau_1,\tau_2,\ldots,\tau_m,X_1,X_2,\ldots, X_{\tau_m})$ for $m\geq 1$ then 
\begin{align*}
& \P\left( \{X_{n+\tau_m}-X_{\tau_m} \}_{n\geq 0} \in A, \, \{\tau_{m+k}-\tau_m\}_{k\geq 1} \in B \, | \, \mathcal{G}_m \right) \\
&\qquad = \P\left( \{X_{n+\tau_1}-X_{\tau_1} \}_{n\geq 0} \in A, \, \{\tau_{1+k}-\tau_1\}_{k\geq 1} \in B\right),
 \end{align*}
 for any Borel measurable sets $A \subset \R^{\Z_+}$ and $B \subset \N^{\Z_+}$. 
That is, the random times $\tau_m, m\ge 1$, split the process into independent pieces, and these pieces are i.i.d.\ after time $\tau_1$.
 We call the random variables $\{\tau_n\}_{n\ge 1}$  {\it regeneration times} for the regenerative process $\{X_n\}_{n\ge 0}$. 
Examples of regenerative process include:
\begin{enumerate}[i)]
 \item Sums $X_n=\sum_{i=1}^n\xi_i$ of iid random variables $(\xi_i)_{i\in\N}$, where we take $\tau_k=k$. 
 \item Additive functionals $X_n = \sum_{i=1}^n f(\zeta_i)$ of a recurrent, irreducible Markov chain $\{\zeta_i\}_{i\geq 0}$ on a countable state space $\mathcal{S}$. In this case one defines $\tau_n$ to be the $n$-th visit of the Markov chain to a fixed state $o \in \mathcal{S}$ in the state space of the Markov chain. 
 \item A ballistic random walk $(X_n)_{n\in\N}$ in a random environment under the annealed measure, where $(\tau_n)_{n\in\N}$ are defined to be the non-backtracking times in a fixed direction of transience (see Section \ref{sec:RWRE} for definitions of these terms).
\end{enumerate}

Since a regenerative process has the same law after any regeneration time $\tau_m$ with $m\geq 1$, and since this law may be different from the law of the process after time $\tau_0 = 0$, it is convenient to denote by $\bP$ the law of the process after a regeneration time. That is, 
\begin{equation}\label{bPdef}
 \bP\left( \{X_n\}_{n\geq 0} \in A, \, \{\tau_k\}_{k\geq 1} \right)
 = \P\left( \{X_{\tau_1+n}-X_{\tau_1} \}_{n\geq 0} \in A, \, \{\tau_{1+k}-\tau_1\}_{k\geq 1} \in B \right).
\end{equation}
We will denote expectations with respect to the measures $\P$ and $\bP$ by $\E$ and $\bE$, respectively. 

For a regenerative process $(X_n)_{i\in\N}$, we let $X_0=0$ and denote the increments by $\xi_i:=X_i-X_{i-1}$ for $i\in\N$.
If $\bE[ \sum_{i=1}^{\tau_1} |\xi_i| ] < \infty$ then it follows from standard arguments that 
\begin{equation}\label{eq:regLLN}
 \lim_{n\to\infty} \frac{X_n}{n} = \frac{\bE[X_{\tau_1}]}{\bE[\tau_1]} =: \mu, \quad \P\text{-a.s.}
\end{equation}
Moreover, if $\bE[\tau_1] < \infty$ and $\bE\left[ \left( \sum_{i=1}^{\tau_1} |\xi_i-\mu| \right)^2 \right] < \infty$ 
then a CLT holds for the sums of the regenerative sequence. That is, if $\Phi(t)$ is the standard normal distribution function, then 
\begin{equation}\label{eq:regCLT}
 \lim_{n\to\infty} \P\left( \frac{X_n - n\mu}{\s \sqrt{n}} \leq t \right) = \Phi(t) 
\quad \forall t \in \R, 
\quad \text{where } 
\s^2 := \frac{\bE[(X_{\tau_1} - \tau_1 \mu)^2]}{\bE[\tau_1]} > 0.
\end{equation}

The main result in this paper is the following theorem which gives polynomial rates of convergence for the regenerative CLT in \eqref{eq:regCLT} under appropriate moment assumptions. 
\begin{thm}\label{th:BEregCLT}
 Assume for some $\d \in (0,1]$ that 
 \[
  \bE[\tau_1^{2+\d}] < \infty, \quad \bE\left[ \left( \sum_{i=1}^{\tau_1} |\xi_i| \right)^{2+\d} \right] < \infty, 
\quad 
 \E[ \tau_1^\d ] < \infty, \quad \text{and}\quad
 \E[ X_{\tau_1}^\d ] < \infty,
\]
then 
there exists a constant $C<\infty$ such that 
\begin{equation}\label{BEuniform}
 \sup_{t \in \R} \left| \P\left( \frac{X_n - n\mu}{\s \sqrt{n}} \leq t \right) - \Phi(t) \right| \leq \frac{C}{n^{\d/2}}, \quad \forall n\geq 1, 
\end{equation}
where $\mu$ and $\s$ are defined as in \eqref{eq:regLLN} and \eqref{eq:regCLT}. 
\end{thm}

Theorem \ref{th:BEregCLT} generalizes several known results. First of all, for i.i.d.\ sequences (i.e., when $\tau_k \equiv k$) the conclusion of Theorem \ref{th:BEregCLT} is the classical Berry-Esseen Theorem \cite{berryCLT,esseenCLT}. 
For regenerative sequences, the results of Theorem \ref{th:BEregCLT} for the case $\d = 1$ were proved by Bolthausen\footnote{In \cite{bo80}, the results were for additive functionals of positive recurrent Markov chains. However, the proofs in \cite{bo80} only use the regenerative structure of positive recurrent Markov chains and thus go through without change for regenerative processes.} in \cite{bo80}.
Some of the techniques introduced by Bolthausen were then used in \cite{hAECLT,MLTSRS} to obtain asymptotic expansions of the CLT (i.e., identifying lower order terms in the CLT error beyond the Berry-Esseen rates) under higher moment assumptions. The results of this paper extend the results of \cite{bo80} in a different direction, obtaining weaker bounds on the rate of decay in the CLT error but under less restrictive moment assumptions.

For i.i.d.\ sequences, the Berry-Esseen Theorem states that the constant $C$ in Theorem \ref{th:BEregCLT} can be given by $\frac{C_\d E[|\xi-\mu|^{2+\d}]}{E[(\xi-\mu)^2]^{1+\d/2}}$ for some absolute constant $C_\d < \infty$ depending only on $\d \in (0,1]$.  
In this paper we are primarily concerned with the polynomial rate of decay and thus we do not compute the constant $C$ explicitly. However, if one examines carefully the proofs in the paper, it can be seen that these show that
\begin{equation}\label{BElimsup}
 \limsup_{n\to\infty} n^{\d/2} \sup_{t \in \R} \left| \P\left( \frac{X_n - n\mu}{\s \sqrt{n}} \leq t \right) - \Phi(t) \right| \leq C' < \infty,
\end{equation}
and that the constant $C'$ can be expressed explicitly in terms of 
certain moments of $\tau_1$, $X_{\tau_1}$, $\sum_{i=1}^{\tau_1} |\xi_i|$ and $(X_{\tau_1} - \mu \tau_1)$ under the measures $\P$ and $\bP$. 
However, since for one of the main applications that we are interested in (random walks in random environments) the moments of $\tau_1$ and $X_{\tau_1}$ cannot be explicitly computed, we focus on the polynomial rate of decay rather than computing explicit uniform upper bounds. 
We note also that \eqref{BElimsup} is sufficent to imply that the uniform upper bound \eqref{BEuniform} holds for some (non-explicit) $C<\infty$, and thus our proof below will focus on proving \eqref{BElimsup} rather than \eqref{BEuniform}. 

\subsection{Outline of the paper}

The remainder of the paper is structured as follows. 
In Section \ref{sec:Markov} we show how Theorem \ref{th:BEregCLT} can be applied to additive functionals of Markov chains satisfying certain mixing conditions and moment bounds, and then in Section \ref{sec:RWRE} we give applications of Theorem \ref{th:BEregCLT} to ballistic RWRE on $\Z^d$ for any $d\geq 1$. 
In both Sections \ref{sec:Markov} and \ref{sec:RWRE} certain applications require Theorem \ref{th:BEregCLT} with $\d<1$, showing the necessity of generalizing the previous results in \cite{bo80}. 
The proof of Theorem \ref{th:BEregCLT} is then given in Sections \ref{sec:SLBE} and \ref{sec:regen}. The general approach of these two sections follows that of \cite{bo80}, but certain parts need to be adapted due to the weaker moment assumptions. In particular, the main result of Section \ref{sec:SLBE} (Theorem \ref{thm:semi-local-LLT}) is a semi-local Berry-Esseen estimate for sums of two-dimensional i.i.d.\ random variables that is quite technical and required significant work to generalize the corresponding semi-local Berry-Esseen estimates in \cite{bo80}. 
Finally, in Section \ref{sec:disc} we again consider the rates of convergence of CLTs for RWRE, comparing the results of this paper with other recent results and posing a few open questions regarding CLTs for RWRE which cannot be handled using the regenerative methods in this paper.

Throughout the paper we will use notation such as $c,c',C,C'$ to denote generic positive constants whose specific values are not important and which can change from one line to the next. 
Specific constants whose value remains the same throughout the paper are denoted by numbered subscripts like $c_0,c_1,C_0,C_1$. 
When we wish to denote the dependence of a constant on a particular parameter we will use subscript such as $C_\e$ or $C_f$ to denote this dependence. 

\section{Application to additive functionals of Markov chains}\label{sec:Markov}

As a first application of Theorem \ref{th:BEregCLT} we consider additive functionals of Markov chains. Let $\zeta = \{\zeta_n\}_{n\geq 0}$ be an irreducible, positive recurrent Markov chain on a countable state space $\mathcal{S}$ and let $X_n = \sum_{i=1}^n f(\zeta_i)$ for some function $f:\mathcal{S} \to \R$. 
For a probability distribution $\nu$ on $\mathcal{S}$ we will denote the law of the Markov chain with initial condition $\zeta_0 \sim \nu$ by $\P_\nu$. If we start at a fixed point $\zeta_0 = x \in \mathcal{S}$ then we will use $\P_x$ in place of $\P_{\delta_x}$. 
Central limit theorems have been proved for additive functionals of Markov chains under a number of conditions
(see for instance \cite{cMCSTP,jMarkovCLT,mtMCSS}).
We will be interested here in conditions for a CLT which are given in terms of the \emph{strong mixing coefficients} of the Markov chain, 
\[
 \a(n) = \sup_m \sup_{A \in \s(\zeta_i, \, i\leq m)} \sup_{B \in \s(\zeta_i, \, i \geq m+n)} |\P_\pi(A \cap B) - \P_\pi(A)\P_\pi(B)|. 
\]
For positive recurrent, aperiodic Markov chains it is known that $\lim_{n\to\infty} \a(n) = 0$ \cite[p. 195]{rMP}. 
The following Theorem, which is a direct application of \cite[Theorem 18.5.3]{ilISS}, shows that 
if the strong mixing coefficients decay fast enough then there is a CLT for the additive functional $X_n$. 
\begin{thm}[Theorem 18.5.3 in \cite{ilISS}]\label{th:mixingCLT}
 Let $X_n = \sum_{i=1}^n f(\zeta_i)$, where $\{\zeta_i\}_{i\geq 0}$ is an irreducible, positive recurrent Markov chain on a countable state space $\mathcal{S}$ with stationary distribution $\pi$. Assume that for some $p \in (2,\infty]$
 \begin{enumerate}\renewcommand{\theenumi}{\roman{enumi}}
  \item $f \in L^p(\mathcal{S},\pi)$,
  \item and $\sum_{n\geq 1} \a(n)^{\frac{p-2}{p}} < \infty$, where $\a(n)$ are the strong mixing coefficients. 
 \end{enumerate}
Then, 
\begin{equation}\label{mufsfdef}
 \mu_f := \E_\pi[ f(\zeta_0) ] < \infty 
\quad\text{and}\quad 
\s_f^2 := \Var_\pi(f(\zeta_0)) + 2 \sum_{k=1}^\infty \Cov_\pi(f(\zeta_0),f(\zeta_k)) < \infty,
\end{equation}
and if $\s_f>0$ then 
\begin{equation}\label{MarkovCLT}
 \lim_{n\to\infty} \P_\pi\left( \frac{X_n-\mu_f n}{\s_f \sqrt{n}} \leq x \right) = \Phi(x), \quad \forall x \in \R. 
\end{equation}
\end{thm}
The main goal of this Section is to show how Theorem \ref{th:BEregCLT} allows us to obtain quantitative bounds on the polynomial rate of convergence for the CLT in \eqref{MarkovCLT} under slightly stronger assumptions on the strong mixing coefficients.

\begin{thm}\label{th:BEmixing}
 Let $X_n = \sum_{i=1}^n f(\zeta_i)$, where $\zeta$ is an irreducible, positive recurrent Markov chain on a countable state space $\mathcal{S}$ with stationary distribution $\pi$. Assume 
for some $p \in (2,\infty]$ and $\l > \frac{2}{p-2}$ that
 \begin{enumerate}\renewcommand{\theenumi}{\roman{enumi}}
  \item $f \in L^p(\mathcal{S},\pi)$ 
  \item and $\sum_{n\geq 1} n^\l \a(n) < \infty$, where $\a(n)$ are the $\a$-mixing coefficients. 
 \end{enumerate}
Then $\mu_f$ and $\s_f$ defined in \eqref{mufsfdef} are finite, and if $\s_f>0$ and the initial distribution $\nu$ of the Markov chain is bounded by some multiple of the stationary distribution $\pi$,
then there exists a constant $C>0$ such that
\[
  \sup_x \left| \P_\nu \left( \frac{X_n - \mu_f n}{\s_f \sqrt{n}} \leq x \right) - \Phi(x) \right| 
 \leq 
 \begin{cases}
  C n^{-\min\{ \frac{\l(p-2)-2}{2(\l+1+p)}, \frac{1}{2} \} } & \text{if } 2<p<\infty \\
  C n^{-\min\{ \frac{\l}{2}, \frac{1}{2} \} } & \text{if } p = \infty. 
 \end{cases}
\]
\end{thm}

\begin{rem}
 The assumptions on the mixing coefficients in Theorem \ref{th:BEmixing} are only slightly stronger than in Theorem \ref{th:mixingCLT}. Indeed, if $\sum_n n^\l \a(n)$ for some $\l > \frac{2}{p-2}$ then 
\[
 \sum_n \a(n)^{\frac{p-2}{p}} 
 = \sum_n \left( n^\l \a(n) \right)^{\frac{p-2}{p}} n^{-\frac{\l(p-2)}{p}} 
\leq \left( \sum_n n^\l \a(n) \right)^{\frac{p-2}{p}} \left( \sum_n n^{-\frac{\l(p-2)}{2}}  \right)^{\frac{2}{p}} < \infty. 
\]
Conversely, since $\a(n)$ is non-increasing it can be shown that if $\sum_n \a(n)^{\frac{p-2}{p}} < \infty$ then 
$ \sum_n n^\l \a(n) < \infty$ for any $\l < \frac{2}{p-2}$. 
\end{rem}

\begin{rem}\label{rem:orates}
 Theorem \ref{th:BEmixing} extends another result of Bolthausen from \cite{bo80}. 
In \cite{bo80} it was shown that the optimal $\bigo(1/\sqrt{n})$ rates of convergence for the CLT of $X_n$ hold when $p>3$ and $\l \geq \frac{p+3}{p-3}$ (including the case when $p=\infty$ and $\l \geq 1$). 
These are exactly the cases in which Theorem \ref{th:BEmixing} gives $\bigo(1/\sqrt{n})$ rates of convergence. In contrast, Theorem \ref{th:BEmixing} gives slower polynomial rates of convergence 
when either 
\begin{equation}\label{lpcases}
 \text{(i) } p \in (2,3] \text{ and } \l > \frac{2}{p-2}, \quad \text{or} \quad \text{(ii) } p > 3 \text{ and } \l \in \left(\frac{2}{p-2},\frac{p+3}{p-3} \right),
\end{equation}
where in the second case we are including $p=\infty$ and $\l \in (0,1)$. 
\end{rem}

\begin{proof}
As noted in Remark \ref{th:BEmixing}, due to the results in \cite{bo80} we need only give the proof of Theorem \ref{th:BEmixing} when $\l>0$ and $p>2$ satisfy one of the two cases in \eqref{lpcases}.  
We will show that in these cases one can find a regenerative structure to apply Theorem \ref{th:BEregCLT} with 
\begin{equation}\label{deltalp}
 \d = \begin{cases} \frac{ \l(p-2)-2}{\l+1+p}  &\text{if } p<\infty \\ \l & \text{if } p=\infty. \end{cases}
\end{equation}
Note that the conditions on $\l$ and $p$ in \eqref{lpcases} imply that $\d$ defined in this way satisfies $\d \in (0,1)$.

To obtain a regenerative structure for the additive functional $X_n = \sum_{i=1}^n f(\zeta_i)$, fix an arbitrary state $o \in \mathcal{S}$ and define the regeneration times to be the successive return times of the Markov chain to $o$. That is, $\tau_0 = 0$ and $\tau_k = \inf\{n > \tau_{k-1} : \, \zeta_n = o \}$ for $k\geq 1$.
In this case, the distribution $\bP$ defined in \eqref{bPdef} is simply $\P_o$ 
and thus since we are assuming that the initial distribution $\nu$ is bounded by a multiple of the stationary distribution
Theorem \ref{th:BEregCLT} will give rates of convergence for a CLT of $X_n$ if 
\begin{equation}\label{MCmoments}
 \E_o[\tau_1^{2+\d}] < \infty , \quad \E_o\left[\left( \sum_{i=1}^{\tau_1} |f(\zeta_i)| \right)^{2+\d} \right] < \infty, \quad 
\E_\pi[\tau_1^\d] < \infty, \quad \text{and}\quad \E_\pi\left[ \left| \sum_{i=1}^{\tau_1} f(\zeta_i) \right|^\d \right] < \infty,
\end{equation}
with $\d \in (0,1)$ defined as in \eqref{deltalp}. 

It was shown in \cite[Theorem 2]{bo80} that the mixing condition $\sum_n n^\l \a(n)<\infty$ implies that  $\E_o[\tau_1^{2+\l}]<\infty$ and therefore also that $\E_{\pi}[\tau_1^{1+\l}] < \infty$. 
Since it can easily be checked that $\frac{\l(p-2)-2}{\l+1+p} \leq \l$, it follows that the first and third conditions in \eqref{MCmoments} hold. 

In the case when $p=\infty$, the function $f$ is then bounded and the second and fourth conditions in \eqref{MCmoments} are finite whenever the first and third conditions are finite. Therefore, for the remainder of the proof we will assume that $p \in (2,\infty)$. 
To verify the second condition in \eqref{MCmoments} in this case, note that 
\begin{align*}
 \E_o\left[ \left( \sum_{i=1}^{\tau_1} |f(\zeta_i)| \right)^{2+\d} \right]
 &=  \E_o\left[ \left( \sum_{i=1}^{\tau_1} |f(\zeta_i)| \right)^{p \frac{2+\d}{p}} \right] \\
 &\leq \E_o\left[ \left( \tau_1^{p-1} \sum_{i=1}^{\tau_1} |f(\zeta_i)|^p \right)^{\frac{2+\d}{p}} \right] \\
 &\leq \E_o\left[ \tau_1^{\frac{(p-1)(2+\d)}{p-2-\d}}  \right]^{\frac{p-2-\d}{p}} \E_o\left[ \sum_{i=1}^{\tau_1} |f(\zeta_i)|^p \right]^{\frac{2+\d}{p}} \\
 &= \E_o\left[ \tau_1^{2+\l} \right]^{\frac{p-1}{\l+1+p}} \E_o\left[ \sum_{i=1}^{\tau_1} |f(\zeta_i)|^p \right]^{\frac{2+\l}{\l+1+p}},
\end{align*}
where the second inequality follows from H\"older's inequality since $\frac{p}{2+\d} = \frac{\l+1+p}{2+\l} > 1$ and the last equality follows from the definition of $\d$ in \eqref{deltalp}. 
We have already shown that the first expectation in the last line is finite, and the second expectation is also finite since $f \in L^p(\mathcal{S},\pi)$ and 
\[
 \E_o\left[ \sum_{i=1}^{\tau_1} |f(\zeta_i)|^p \right]
= \sum_{x \in \mathcal{S}} E_o\left[ \sum_{i=1}^{\tau_1} \ind{\zeta_i = x} \right] |f(x)|^p
= \E_o[\tau_1] \sum_{x \in \mathcal{S}} \pi(x) |f(x)|^p.
\]

Finally, for the second condition in \eqref{MCmoments}, in the proof of Lemma 1 on page 61 of \cite{bo80} it was shown that 
\[
 \E_\pi\left[ \sum_{i=1}^{\tau_1} |f(\zeta_i)| \right] 
 \leq 2 \pi(o) \E_o\left[ \left( \sum_{i=0}^{\tau_1 - 1} |f(\zeta_i)| \right)^2 \right] + 2 \pi(o)\E_o[\tau_1^2] + \max\{|f(o)|,1\},
\]
and the terms on the right are all finite by the arguments above. Since $\delta < 1$ this is more  than enough to verify the first second condition in \eqref{MCmoments}. 
\end{proof}

\begin{rem}
 For Harris recurrent Markov chains on more general state spaces, under a certain regularity assumption Nummelin \cite{nSplitting} developed a ``splitting'' technique which allows one to construct a related Markov chain which does have regeneration times. The proof of Theorem \ref{th:BEmixing} can be extended to such Harris recurrent Markov chains using this splitting technique in the same manner as was done by Bolthausen in \cite{bBETHRMC} in the case when $p>3$ and $\l \geq \frac{p+3}{p-3}$. 
\end{rem}

\begin{rem}
 The proof of the CLT for $X_n = \sum_{i=1}^n f(\zeta_i)$ using the regenerative structure as in the proof above shows that $\mu_f$ and $\s_f$ as defined in \eqref{mufsfdef} must also have the alternative expressions 
\begin{equation}\label{mufsfdef2}
 \mu_f = \frac{E_o[\sum_{i=1}^{\tau_1} f(\zeta_i) ]}{E_o[\tau_1]} \quad \text{and}\quad \s_f^2 = \frac{\E_o\left[ \left( \sum_{i=1}^{\tau_1} f(\zeta_i) - \mu_f \tau_1 \right)^2 \right]}{\E_o[\tau_1]}.  
\end{equation}
The equality of the expressions in \eqref{mufsfdef} and \eqref{mufsfdef2} can also be verified more directly using the representation of the stationary distribution $\pi(x) = \frac{1}{\E_o[\tau_1]} \E_o\left[ \sum_{i=1}^{\tau_1} \ind{\zeta_i = x} \right]$.
\end{rem}

\section{Application to RWRE: Annealed CLT rates}\label{sec:RWRE}

In this section we will show how the results of Theorem \ref{th:BEregCLT} can be applied to certain non-Markovian random walks.
For simplicity we will restrict ourselves to nearest neighbor RWRE, though clearly the same arguments will apply to other non-Markovian random walks with a similar regeneration structure and known bounds on the moments of regeneration times (e.g. excited random walks \cite{brCLTERW,kzPNERW}).

We begin by recalling the model of random walks in random environments. 
For nearest-neighbor RWRE on $\Z^d$, an \emph{environment} $\w$ is a collection of probability distributions on $\mathcal{E}_d = \{\mathbf{z} \in \Z^d: |\mathbf{z}|=1\}$ indexed by the vertices of $\Z^d$. That is, $\w=\{\w_\mathbf{x}(\mathbf{z})\}_{\mathbf{x} \in \Z^d, \mathbf{z} \in \mathcal{E}_d}$ such that $\w_{\mathbf{x}}(\mathbf{z}) \geq 0$ and $\sum_{\mathbf{z} \in \mathcal{E}_d} \w_\mathbf{x}(\mathbf{z}) = 1$ for every $\mathbf{x} \in \Z^d$. Given an environment $\w$, a random walk in the environment $\w$ is a Markov chain $\{\mathbf{X}_n\}_{n\geq 0}$ on $\Z^d$ with law $P_\w$ given by 
\[
 P_\w( \mathbf{X}_0 = \mathbf{0}) = 1
\quad \text{and}\quad 
P_\w( \mathbf{X}_{n+1} = \mathbf{x}+\mathbf{z} \, |\, X_n = \mathbf{x}) = \w_\mathbf{x}(\mathbf{z}), \quad \forall \mathbf{x} \in \Z^d, \, \mathbf{z} \in \mathcal{E}_d, \, n \geq 0. 
\]
A random walk in a random environment is then obtained by first choosing an environment $\w$ randomly according to some fixed probability distribution $P$ on the space of environments and then running a random walk in that fixed environment.
In general it is assumed that the distribution on environments $P$ is ergodic under spatial shifts of $\Z^d$, but for this paper we will adopt the common assumption that the environment is i.i.d.\ -- that is, the family $\{\w_{\mathbf{x}}(\cdot)\}_{\mathbf{x} \in \Z^d}$ of transition probabilities indexed by the vertices of $\Z^d$ is i.i.d.\ under the distribution $P$ on environments.
The distribution $P_\w$ of the walk conditioned on the environment $\w$ is called the \emph{quenched} law, while the distribution 
\begin{equation}\label{annealeddef}
 \P(\cdot) = E\left[ P_\w(\cdot) \right], 
\end{equation}
where both the environment and the walk are random is called the \emph{annealed} (or averaged) law of the RWRE. 
Note that in \eqref{annealeddef} and below $E[\cdot]$ will denote expectation with respect to the distribution $P$ on environments. Expectations with respect to the quenched and annealed laws on the RWRE will be denoted by $E_\w[\cdot]$ and $\E[\cdot]$ respectively.  

While the (multidimensional) Central Limit Theorem implies that classical simple random walks on $\Z^d$ always have a Gaussian limiting distributions under diffusive scaling, random walks in random environments (RWRE) on $\Z^d$ are much more difficult to study and can have limiting distributions which are non-Gaussian (see for instance \cite{kksStable,sRecurrent,bcFKRWRC}).
Nonetheless, there are sufficient conditions for the distribution on the environment which ensure that a CLT holds for the RWRE.

Our main goal in this section is to consider some ballistic (non-zero limiting speed) RWRE for which a CLT is known to hold under the annealed measure and to prove polynomial rates of convergence for this CLT.
While the limiting distributions of RWRE have been studied quite extensively, there has been up until recently very few results giving quantitative bounds on the rates of convergence. In particular, we are only aware of two such prior results for RWRE \cite{mQCLTRWRC,apQCLTrates}. 
Our results below differ from both of these in the following ways. The results in \cite{mQCLTRWRC} considered the random conductance model while our results are applied to certain RWRE in i.i.d.\ environments. Also, the results in \cite{apQCLTrates} gave bounds on the polynomial rate  of convergence for the \emph{quenched} CLT of one-dimensional RWRE while we consider in this paper the rates of convergence for the \emph{annealed} CLT and apply to certain multidimensional RWRE as well. 
A more in depth discussion of the relation between the quenched and annealed rates of convergence for RWRE is given at the end of this paper in Section \ref{sec:disc}.

To apply the results of Theorem \ref{th:BEregCLT} to RWRE, we need to first review the appropriate concepts of regeneration times for RWRE.
If $\{\mathbf{X}_n\}_{n\geq 0}$ is a RWRE on $\Z^d$ and $\mathbf{u} \in S^{d-1} = \{\mathbf{z} \in \R^d: |\mathbf{z}| = 1\}$ is a fixed direction, then, 
setting\footnote{In this definition we are using the convention that $\inf \emptyset = \infty$; that is, if $\tau_{\mathbf{u},k} = \infty$ for some $k$ then $\tau_{\mathbf{u},k+1}$ is taken to be $\infty$ also.}
\begin{equation}\label{regseq}
 \tau_{\mathbf{u},0}=0, \quad \tau_{\mathbf{u},k} = \inf\left\{ n> \tau_{\mathbf{u},k-1}: \sup_{m<n} \mathbf{X}_m \cdot \mathbf{u}  < \mathbf{X}_n \cdot \mathbf{u}  \leq \inf_{m\geq n} \mathbf{X}_m \cdot \mathbf{u}  \right\}, \quad k\geq 1,
\end{equation}
it is known that on the event $A_{\mathbf{u}} = \{\lim_{n\to\infty} X_n \cdot \mathbf{u} = \infty \}$ (that is, when the RWRE is transient in direction $\mathbf{u}$),  the random variables $\{\tau_{\mathbf{u},k}\}_{k\ge 1}$ are almost surely finite \cite{szLLN}, and they are regeneration times for the RWRE under the annealed law $\P$.
 Moreover, the regeneration times reveal the following i.i.d.\ structure within the RWRE: under the conditional measure $\P( \cdot \, | \, A_{\mathbf{u}} )$ the sequence of the sections of the path of the walk between regeneration times
\[
 \left\{ \left( (\mathbf{X}_m - \mathbf{X}_{\tau_{\mathbf{u},k}})_{\tau_{\mathbf{u},k}\leq m \leq \tau_{\mathbf{u},k+1}}, \, \tau_{\mathbf{u},k+1} - \tau_{\mathbf{u},k} \right) \right\}_{k\geq 0} 
\]
is independent for $k\geq 0$ and identically distributed for $k\geq 1$. 
With this i.i.d.\ structure, the following results are known. 
\begin{itemize}
 \item \textbf{LLN \cite{szLLN}:} If $\E[\tau_{\mathbf{u},2}-\tau_{\mathbf{u},1}] < \infty$, then $\lim_{n\to\infty} \frac{\mathbf{X}_n}{n} = \mathbf{\bvp} \neq \mathbf{0}$, almost surely, where 
 \begin{equation}\label{bvpdef}
  \bvp = \frac{ \E[ \mathbf{X}_{\tau_{\mathbf{u},2}} - \mathbf{X}_{\tau_{\mathbf{u},1}} ] }{ \E[ \tau_{\mathbf{u},2} - \tau_{\mathbf{u},1} ] }.
 \end{equation}
  \item \textbf{CLT \cite{sSlowdown}:} If  
$\E[(\tau_{\mathbf{u},2} - \tau_{\mathbf{u},1} )^2] < \infty$ then $\frac{X_n - n \bvp}{\sqrt{n}}$ converges in distribution under the annealed law $\P$ to a $d$-dimensional Normal distribution with zero mean and covariance matrix
  \begin{equation}\label{RWREcov}
   \Sigma = \frac{1}{\E[\tau_{\mathbf{u},2} - \tau_{\mathbf{u},1}]} \E\left[ \left( \mathbf{X}_{\tau_{\mathbf{u},2}} - \mathbf{X}_{\tau_{\mathbf{u},1}} - (\tau_{\mathbf{u},2} - \tau_{\mathbf{u},1})\bvp \right) \left( \mathbf{X}_{\tau_{\mathbf{u},2}} - \mathbf{X}_{\tau_{\mathbf{u},1}} - (\tau_{\mathbf{u},2} - \tau_{\mathbf{u},1})\bvp \right)^T \right]. 
  \end{equation}
 \end{itemize}
 
Since Theorem \ref{th:BEregCLT} gives rates of convergence for a one-dimensional CLT, we can only apply this to one-dimensional projections of a multidimensional RWRE. 
To this end, suppose that there is a direction $\mathbf{u} \in S^{d-1}$ such that $\P(A_{\mathbf{u}}) = 1$. 
Then for any other direction $\mathbf{w} \in S^{d-1}$ we can apply Theorem \ref{th:BEregCLT} to the sequence 
$\mathbf{X}_n \cdot \mathbf{w}$. 
\begin{thm}\label{th:RWRErates}
 Let $\mathbf{X}_n$ be a $d$-dimensional RWRE, and let $\mathbf{u} \in S^{d-1}$ be such that 
 \begin{equation}\label{taumb}
  \E\left[ (\tau_{\mathbf{u},2} - \tau_{\mathbf{u},1})^{2+\d} \right] < \infty
  \quad\text{and}\quad 
  \E\left[ \tau_{\mathbf{u},1}^\d \right] < \infty
 \end{equation}
 for some $\d \in (0,1]$. 
Then, there exists a constant $C<\infty$ such that for any $\mathbf{w} \in S^{d-1}$, 
 \[
  \sup_{x \in \R} \left| \P\left( \frac{(\mathbf{X}_n - n\bvp)\cdot \mathbf{w}}{\s_{\mathbf{w}} \sqrt{n}} \leq x \right) - \Phi(x) \right| \leq \frac{C}{n^{\d/2}}, 
 \]
 where $\bvp$ is as in \eqref{bvpdef} and $\s_{\mathbf{w}}^2 = \mathbf{w}^T \Sigma \mathbf{w}$ where $\Sigma$ is the covariance matrix in \eqref{RWREcov}. 
\end{thm}
\begin{rem}
The following remarks are in order regarding the moment assumptions \eqref{taumb} in Theorem \ref{th:RWRErates}. 
\begin{itemize}
\item  Since the RWRE is a nearest neighbor walk, the random variables 
$\xi_n = (\mathbf{X}_n - \mathbf{X}_{n-1}) \cdot \mathbf{w}$ have
$|\xi_n| \leq 1$ and so the moment bounds in \eqref{taumb} are enough to satisfy the assumptions of Theorem \ref{th:BEregCLT}. 
\item For one-dimensional RWRE, it can be shown under mild ballisticity condition that the requirement \eqref{taumb} is equivalent to $\E[\tau_1^{2+\delta}]<\infty$. See Proposition~\ref{1dregmb}.
\item When the dimension $d\ge 2$, for {\it uniformly elliptic} and ballistic environment,  it is conjectured that all moments of the regeneration times are finite. However, this is not true when the ballistic environment is only assumed to be {\it elliptic}. See the following for more detailed comments.
\end{itemize}
\end{rem}

Theorem \ref{th:RWRErates} reduces the problem of obtaining rates of convergence for the annealed CLT to computing certain moment bounds of the regeneration times. 
For multidimensional RWRE, a great deal of effort has gone into obtaining improved conditions under which moment bounds on regeneration times can be obtained and we will review the best known conditions here, though the full picture is not yet complete. 

\begin{enumerate}
 \item \textbf{Uniformly elliptic environments.}
A nearest neighbor RWRE is called \emph{uniformly elliptic} if there exists a constant $c>0$ such that $P(\w_0(\mathbf{z}) \geq c ) = 1$ for all $|\mathbf{z}| = 1$; that is, the transition probabilities in all directions are uniformly bounded away from zero.  
For uniformly elliptic RWRE, a number of conditions have been shown to imply that $\E[\tau_{\mathbf{v},1}^p] < \infty$ for all $p < \infty$ where $\mathbf{v}\neq \mathbf{0}$ is the limiting speed; these conditions include Kalikow's condition \cite{sSlowdown}, Sznitman's conditions $(T)$, $(T')$ and $(T)_\gamma$ \cite{sConditionT,sEffective}, and the Polynomial condition $(P)$ introduced by Berger, Drewitz, and Ram\'irez \cite{bdrEPBC}. 

We refer the interested reader to the above references for the exact statement of these conditions and simply note that the weakest condition is the polynomial condition $(P)$ and that this condition is ``effective'' in the sense that it can be verified by computing certain exit probabilities of the RWRE from a large but finite multidimensional box. 
We also note that all of the known conditions implying ballisticity (non-zero limiting speed) for uniformly elliptic RWRE imply moments of all orders for the regeneration times. In fact it is conjectured that for uniformly elliptic RWRE in dimension $d\geq 2$ that 
$\P(A_{\mathbf{u}})=1$ (i.e., transience in direction $\mathbf{u}$)
implies that $\E[ \tau_{\mathbf{u},1}^p ] < \infty$ for all $p<\infty$. 
This is in contrast to what is known for one-dimensional RWRE (see Proposition \ref{1dregmb} below) and for multidimensional RWRE which are not uniformly elliptic. 

 \item \textbf{Elliptic environments.} A nearest neighbor RWRE is called \emph{elliptic} if $P(\w_0(\mathbf{z}) > 0 ) = 1$ for all $|\mathbf{z}| = 1$; that is, the transition probabilities in all directions are non-zero but not necessarily uniformly bounded away from zero. 
In \cite{brsSECBB} and \cite{fkLTRWRE}, checkable ellipticity conditions are given which together with the polynomial condition $(P)$ imply the finiteness of certain moments of the regeneration times. 
Moreover, these papers also give explicit examples of elliptic RWRE which satisfy condition $(P)$ or even the stronger Kalikow's condition but which do not have all moments of regeneration times finite. In particular, for i.i.d.\ Dirichlet random environments there are certain choices of the parameters for which the results in \cite{brsSECBB} show that the regeneration times have infinite third moment but finite $(2+\d)$ moments for some $\d \in (0,1)$.
\end{enumerate}

\subsection{One-dimensional RWRE}\label{sec:1dRWRE}
The purpose of this subsection is to consider more in depth the annealed CLT rates of convergence for one-dimensional RWRE. In one dimension we are able to obtain more explicit results as a result of the fact that it is possible to give an explicit criterion for what moments of the regeneration times of the RWRE are finite (see Proposition \ref{1dregmb} below). 
Our main result in this subsection (Corollary~\ref{cor:1dratesX}) gives explicit polynomial rates of convergence for the annealed CLTs of both the position and the hitting times of the walk. 

For a RWRE on $\Z$ there is no need to take a projection to apply Theorem \ref{th:BEregCLT} and so we will write $X_n$ for the position of the walk rather than $\mathbf{X}_n$. Also, if the walk is transient, without loss of generality we can assume it is transient to the right and so we need only consider regeneration times to the right and will therefore write $\tau_k$ rather than $\tau_{1,k}$. 

For one-dimensional RWRE in i.i.d.\ environments, much of the behavior of the walk can be explicitly characterized in terms of the distribution of the random variable $\rho = \frac{\w_0(-1)}{\w_0(1)}$. 
In particular, it was shown in \cite{sRWRE,kksStable} that
\begin{itemize}
 \item the random walk is transient to the right if and only if $E[\log \rho] < 0$,
 \item the limiting speed $\vp=\lim_{n\to\infty} \frac{X_n}{n}$ is positive if and only if $E[\rho] < 1$ with the explicit formula $\vp = \frac{1-E[\rho]}{1+E[\rho]}$ for the speed, 
 \item and if $E[\rho^2] < 1$ then annealed CLTs hold both for the position of the walk $X_n$ and the hitting times $T_n = \inf\{ k\geq 0: \, X_k = n\}$. That is, 
\begin{equation}\label{1daCLT}
\lim_{n\to\infty} \P\left( \frac{X_n - n\vp}{\vp^{3/2} \s_0 \sqrt{n}} \leq x \right) = \Phi(x)
\quad \text{and}\quad 
\lim_{n\to\infty} \P\left( \frac{T_n - n/\vp}{\s_0 \sqrt{n}} \leq x \right) = \Phi(x), \quad \forall x \in \R, 
\end{equation}
where $\s_0^2 = E[\Var_\w(T_1)] + \frac{1}{\vp} \Var(E_\w[T_1])$. 
\end{itemize}
In fact the CLTs in \eqref{1daCLT} are a specific case of a more general result on limiting distributions by Kesten, Kozlov, and Spitzer \cite{kksStable}. If the RWRE is transient to the right (i.e., $E[\log \rho]<0$) then, under mild technical assumptions on the distribution on the environment, the limiting distribution depends on a parameter $\kappa>0$ which is the unique positive solution to $E[\rho^\k] = 1$. 
The assumption $E[\rho^2]<1$ is equivalent to $\kappa>2$, and this is the only case where annealed CLTs like \eqref{1daCLT} hold; if $\k\in (0,2)$ then the limiting distribution is not Gaussian and the scaling is not diffusive, while if $\k = 2$ then the limting distributions of $X_n$ and $T_n$ are Gaussian but with logarithmic corrections to the diffusive scaling $\sqrt{n}$. 
When $\k>2$, the following Corollary of Theorems \ref{th:BEregCLT} and \ref{th:RWRErates} gives polynomial rates of convergence for both of the annealed CLTs in \eqref{1daCLT}. 

\begin{cor}\label{cor:1dratesX}
Assume that $E[\log \rho]<0$ and $E[\rho^\kappa] = 1$ for some $\kappa>2$. 
 \begin{enumerate}
  \item If $\k > 3$, then there exists a constant $C<\infty$ such that 
 \[\tag{1a}
  \sup_{x \in \R} \left| \P\left(\frac{X_n - n\vp}{\s_0 \vp^{3/2} \sqrt{n}} \leq x \right) - \Phi(x) \right| \leq \frac{C}{\sqrt{n}}
 \]
 and 
   \[\tag{1b}
   \sup_{x\in \R} \left| \P\left( \frac{T_n - n/\vp}{\s_0 \sqrt{n}} \leq x \right) - \Phi(x) \right| \leq \frac{C}{\sqrt{n}}. 
  \]
  \item If $\k \in (2,3]$, then   for any $\e>0$,
  \[\tag{2a}
  \lim_{n\to\infty} n^{\frac{\k}{2} - 1 - \e} \sup_{x \in \R} \left| \P\left(\frac{X_n - n\vp}{\s_0 \vp^{3/2} \sqrt{n}} \leq x \right) - \Phi(x) \right| = 0 
 \]
 and 
   \[\tag{2b}
   \lim_{n\to\infty} n^{\frac{\k}{2}-1-\e} \sup_{x \in \R} \left| \P\left( \frac{T_n - n/\vp}{\s_0 \sqrt{n}} \leq x \right) - \Phi(x) \right| = 0. 
  \]
 \end{enumerate}
 \end{cor}

 The key to the proof of Corollary~\ref{cor:1dratesX} will be establishing moment bounds for the regeneration times of the RWRE in terms of the parameter $\kappa$. 
As a first step in this direction, the following lemma shows that $\kappa$ determines what moments of hitting times are finite. 
\begin{lem}\label{lem:ETmb}
 Assume that $E[\log\rho] < 0$ and that $E[\rho^\k] = 1$ for some $\k\geq 1$. Then $E[T_1^\gamma] < \infty$ if and only if $\gamma < \k$. 
\end{lem}
\begin{proof}
 It was shown in \cite{dpzTE1D} that $\gamma < \k$ implies that $\E[T_1^\gamma] < \infty$. For the reverse implication, we will use the fact that the quenched expectation of $T_1$ has the explicit formula (see \cite{sRWRE} or \cite{zRWRE}),
\[
 E_\w[T_1] = 1 + 2 \sum_{k=1}^\infty \prod_{x=-k+1}^0 \frac{\w_x(-1)}{\w_x(1)}. 
\]
Therefore, if $\gamma \geq 1$
\begin{align*}
 \E[T_1^\gamma] \geq E\left[ \left( E_\w[T_1] \right)^\gamma \right] 
\geq 2^\gamma E\left[ \sum_{k=1}^\infty  \left(\prod_{x=-k+1}^0 \frac{\w_x(-1)}{\w_x(1)} \right)^\gamma \right] 
= 2^\gamma  \sum_{k=1}^\infty E[\rho^\gamma]^k,
\end{align*}
where we used that the environment is i.i.d.\ in the last equality. If $\gamma \geq \k$, then it follows from Jensen's inequality that $E[\rho^\gamma] \geq E[\rho^\k]^{\gamma/\k} = 1$, and thus the sum on the right above is infinite. 
\end{proof}


The following Proposition shows that the parameter $\k$ also determines what moments of the regeneration times are finite. 
\begin{prop}\label{1dregmb}
 Assume that $E[\log\rho] < 0$ and that $E[\rho^\k]= 1$ for some $\k\geq 1$. Then $\E[\tau_1^\gamma] < \infty$ and $\E[(\tau_2-\tau_1)^\gamma] < \infty$ if and only if $\gamma < \k$. 
\end{prop}
 
 \begin{proof}[Proof of Proposition \ref{1dregmb}]
In the context of one-dimensional RWRE, the measure $\bP$ as defined in \eqref{bPdef} for the regenerative sequence $X_n$ is the same as $\P(\cdot \, | \, T_{-1} = \infty)$. 
Therefore, 
\[
 \E[ (\tau_2 - \tau_1)^\gamma ] = \bE[\tau_1^\gamma] = \frac{\E[\tau_1^\gamma \ind{T_{-1}=\infty}]}{\P(T_{-1}=\infty)},  
\]
which implies that $\P(T_{-1}=\infty) \bE[\tau_1^\gamma] \leq \E[\tau_1^\gamma]$. 
Since $\P(T_{-1}=\infty)>0$ when the RWRE is transient to the right, it follows that it is enough to prove that $\E[\tau_1^\gamma] < \infty$ if $\gamma < \k$
and $\bE[\tau_1^\gamma] = \infty$ if $\gamma \geq \k$. 

To prove that $\E[\tau_1^\gamma] < \infty$ when $\gamma < \kappa$, by decomposing according the the location of the walk at the first regeneration time, we obtain that for any $\varepsilon>0$,
\begin{align*}
 \E[\tau_1^\gamma]
= \sum_{n=1}^\infty \E\left[ (T_n)^\gamma \ind{X_{\tau_1} = n} \right] 
&\leq \sum_{n=1}^\infty \E\left[ (T_n)^{\gamma(1+\e)} \right]^{\frac{1}{1+\e}} \P(X_{\tau_1} = n)^{\frac{\e}{1+\e}} \\
&\leq \E\left[T_1^{\gamma(1+\e)} \right]^{\frac{1}{1+\e}}  \sum_{n=1}^\infty n^{\gamma \vee \left( \frac{1}{1+\e} \right) } \P(X_{\tau_1} = n)^{\frac{\e}{1+\e}},
\end{align*}
where the last inequality 
follows either from Minkowski's inequality when $\gamma(1+\e)\geq 1$ or from the subadditivity of $x\mapsto x^{\gamma(1+\e)}$ when $\gamma(1+\e)<1$. 
Since \cite[Prop.\ 2.6]{sConditionT} implies\footnote{In general, the results in \cite{sConditionT} assume that the RWRE is ``uniformly elliptic,'' i.e., that all transition probabilities are uniformly bounded away from zero. However, an examination of the proof of Proposition 2.6 in that paper shows that the uniform ellipticity assumption is not needed there.} that $\P(X_{\tau_1} = n) \leq e^{-c n}$ for some $c>0$, it follows from Lemma \ref{lem:ETmb} that if $\gamma<\kappa$ the right side is finite for $\e>0$ sufficiently small.

To prove that $\bE[\tau_1^\gamma] = \infty$ when $\gamma \geq \k$, note that if $\gamma \geq 1$ then $\P$-almost surely,
\begin{align*}
 \bE[\tau_1^\gamma] &= \lim_{n\to\infty} \frac{1}{n} \sum_{k=1}^n (\tau_k-\tau_{k-1})^\gamma \\
&= \lim_{n\to\infty} \frac{1}{n} \sum_{k=1}^n \left( \sum_{x=X_{\tau_{k-1}}+1}^{X_{\tau_k}} (T_x-T_{x-1}) \right)^\gamma \\
&\geq \lim_{n\to\infty} \frac{1}{n} \sum_{k=1}^n \sum_{x=X_{\tau_{k-1}}+1}^{X_{\tau_k}} (T_x-T_{x-1})^\gamma
= \lim_{n\to\infty} \frac{1}{n} \sum_{x=1}^{X_{\tau_n}} (T_x-T_{x-1})^\gamma
= \bE[X_{\tau_1}] \E[T_1^\gamma]. 
\end{align*}
where in the last equality we used that the sequence $\{T_x-T_{x-1}\}_{x\geq 1}$ is ergodic under the annealed measure \cite{sRWRE}. 
Therefore, if $\gamma \geq \k \geq 1$ it follows from Lemma \ref{lem:ETmb} that $\bE[\tau_1^\gamma] = \infty$. 
 \end{proof}

\begin{proof}[Proof of Corollary~\ref{cor:1dratesX}]
Applying Proposition \ref{1dregmb} to Theorem \ref{th:RWRErates} for any $\d < (\k-2)\wedge 1$, we immediately obtain (1a) and (2a).

The proofs of (2a) and (2b) also follow from Theorem \ref{th:BEregCLT}, but applied to a different regenerative process. 
Represent $T_n = \sum_{k=1}^n \zeta_i$ where $\zeta_i = T_i-T_{i-1}$. Then under the annealed measure $\P$ the sequence 
$(T_n)_{n\ge 1}$
 is a regenerative process with ``regeneration times'' $0=\s_0 < \s_1 < \s_2 <\cdots$ where $\s_k = X_{\tau_k}$ is the position of the walk at the time of the $k$-th regeneration time of the walk. 
Since the crossing times $\zeta_i \geq 1$, to apply Theorem \ref{th:BEregCLT} we need only to check that $\E[ \left( \sum_{i=1}^{\s_1} \zeta_i \right)^\d ] < \infty$ and $\E[ \left( \sum_{i=\s_1+1}^{\s_2} \zeta_i \right)^{2+\d} ] < \infty$ for some $\d \in (0,1]$. 
However, since 
\[
 \sum_{i=\s_{k-1}+1}^{\s_k} \zeta_i = T_{\s_k} - T_{\s_{k-1}} = T_{X_{\tau_k}} - T_{X_{\tau_{k-1}}} = \tau_k-\tau_{k-1},
\]
this is equivalent to checking that $\E[\tau_1^\d]$ and $\E[(\tau_2 - \tau_1)^{2+\d}] < \infty$, and by Proposition \ref{1dregmb} this holds for $\d=1$ if $\k>3$ and for any $\d \in (0,2-\k)$ if $\k \in (2,3]$. 
\end{proof}

\section{A non-uniform semi-local Berry-Esseen bound}\label{sec:SLBE}

Consider a random variable $\vvar=(\avar,\bvar)\in\R^2$ with zero-mean $E[\vvar]=\bf 0$ and a positive-definite covariance matrix 
\[
\Sigma=
\begin{pmatrix}
\Var(\avar)&\Cov(\avar,\bvar)\\
\Cov(\avar,\bvar)&\Var(\bvar)
\end{pmatrix}
=
\begin{pmatrix}
\sigma_1^2&\sigma_{12}\\
\sigma_{12}&\sigma_2^2
\end{pmatrix}>0.
\]
(That is, both eigenvalues of $\Sigma$ are strictly positive.) Let $\vvar_i=(\avar_i,\bvar_i)$, $i\in\N$, denote iid copies of $\vvar$ and
\[
\vsum_n=(\asum_n,\bsum_n):= \left(\sum_{i=1}^n \avar_i, \sum_{i=1}^n \bvar_i \right).
\]
Throughout this section, we assume that almost surely, $\bvar\in\rho+\Z$ for some $\rho\in\R$ and that $\bvar$ has a lattice distribution with span $1$.

By the central limit theorem,  if $E[|\vvar|^2]<\infty$, then
$\vsum_n/\sqrt n$ converges weakly to a two-dimensional normal random variable $\mathcal N=(\mathcal N_1,\mathcal N_2)$ with covariance matrix $\Sigma$. Here $|\vvar|:=\sqrt{\avar^2+\bvar^2}$. Moreover, when $E[|\bvar|^3]<\infty$, the classical local limit theorem (LLT) states that  the probability mass function of $\bsum_n/\sqrt n$ converges to the density of $\mathcal N_2$. See \cite[VII]{petrovRW}. Under weaker moment condition $E[|\bvar|^{2+\delta}]<\infty$ for some $\delta\in(0,1]$,  the following non-uniform estimate of the convergence rate holds for the LLT. See \cite{shev17} and \cite{bcg11}.

\begin{prop}\label{thm:LLT}
Assume that 
$E[|\bvar|^{2+\delta}]<\infty$ for $\delta\in(0,1]$. Writing $y_n:=(y+n\rho)/\sqrt n$ for $y\in\Z$.  
Then
\[
\sup_{y\in\Z}(1+y_n^2)\Abs{P \left(\tfrac{\bsum_n}{\sqrt n}=y_n \right)-\tfrac{1}{\sigma_2 \sqrt{2n\pi}  }e^{-y_n^2/2\sigma_2^2}}
\le Cn^{-\tfrac{1+\delta}{2}},
\]
where the constant $C$ depends only on $\delta$ and $E[|\bvar|^{2+\delta}]$.
\end{prop}

For any positive definite $2\times 2$ matrix, let 
$\gamma_A(\mathbf{x}) = C_A \exp\{ -\mathbf{x}^T A^{-1}\mathbf{x}/2 \}$, $\mathbf{x} \in \R^2$ be the density function of a centered Gaussian with covariance matrix $A$ and let 
\begin{equation}\label{psidef}
 \psi_A(x,y) = \int_{-\infty}^x \gamma_A(t,y) \, dt. 
\end{equation}
The purpose of this section is to generalize Proposition~\ref{thm:LLT} to 
a non-uniform estimate of a {\it semi-local} limit theorem, which is of interest in its own right.
\begin{thm}\label{thm:semi-local-LLT}
Assume that $E[|\vvar|^{2+\delta}]<\infty$ for $\delta\in(0,1]$, then 
\[
\sup_{x\in\mathbb R, y\in\mathbb Z}(1+y_{n}^2)
\bigg|P\left( \tfrac{\asum_n}{\sqrt n}\le x, \tfrac{\bsum_n}{\sqrt n}=y_n \right)-\frac{1}{\sqrt n}\psi_\Sigma(x, y_{n})\bigg|\le Cn^{-(1+\delta)/2}.
\]
\end{thm}

For the case $\delta=1$, Theorem~\ref{thm:semi-local-LLT} was previously obtained by Bolthausen\cite[Theorem~4]{bo80}. 
Our proof follows the main idea of \cite{bo80}, where characteristic functions (ch.f.) are used to express the probabilities. In fact, the term $y_n^2$ comes from second-order derivatives of ch.f.'s.  However, unlike \cite{bo80}, estimates about the third order derivative of ch.f.'s (which were used to bound the difference of the second-order derivatives) are not available because of the lack of moments when $\delta\in(0,1)$. To overcome this difficulty, we will use a Lipschitz-type estimate of the second order derivative of the ch.f.'s. See Proposition~\ref{prop:chf-for-LLT}(c).

In Subsection~\ref{subsec:exp-chf}, we obtain useful estimates of characteristic functions, which will yield an easy proof of Proposition~\ref{thm:LLT} in Subsection~\ref{subsec:pf-llt}. Further, making use of these results, we will prove Theorem~\ref{thm:semi-local-LLT} in Subsection~\ref{subsec:pf-semi-llt}.

\subsection{Estimates of characteristic functions}\label{subsec:exp-chf}
Let $\mathbf{t}=(t_1,t_2)\in\R^2$. We denote the characteristic functions of $\vvar$, $\vsum_n/\sqrt n$ and $\mathcal N$ by $\vphi(\mathbf{t})$, $\lambda_n(\mathbf{t})=\vphi(\mathbf{t}/\sqrt n)^n$ and $\lambda_0(\mathbf{t})=\exp(-\mathbf{t}^T\Sigma \mathbf{t}/2)$, respectively.
\begin{prop}\label{prop:chf-for-LLT}
Assume  $E[|\vvar|^{2+\delta}]<\infty$ for $\delta\in(0,1]$. Then there exist positive constants $\varepsilon, c, C$ depending on $\delta, \Sigma$ and $E[|\vvar|^{2+\delta}]$ such that for any $\mathbf{t}\in\R^2$ with $|t_1|\le \varepsilon\sqrt n$, $|t_2|\le \pi\sqrt n$,
\begin{itemize}
\item[(a)]
$\Abs{\varphi(\tfrac{\mathbf{t}}{\sqrt n})^{n-j}-\lambda_0(\mathbf{t})}
\le 
Cn^{-\delta/2}e^{-c|\mathbf{t}|^2}$,  $\forall j=0,1,2;$
\item[(b)] $\Abs{\frac{\partial^2}{\partial t_2^2}(\lambda_n(\mathbf{t})-\lambda_0(\mathbf{t}))}
\le 
Cn^{-\delta/2}e^{-c|\mathbf{t}|^2}$;
\item[(c)]
Set $\Lambda(\mathbf{t})=\Lambda_n(\mathbf{t}):=\frac{\partial^2}{\partial t_2^2}(\lambda_n(\mathbf{t})-\lambda_0(\mathbf{t}))$. Then there exists a constant $c_0>0$ such that
\[
\Abs{\Lambda(t_1,t_2)-\Lambda(0,t_2)+\vphi(0,\tfrac{t_2}{\sqrt n})^{n-1}E[\bvar^2(e^{i\mathbf{t}\cdot \vvar/\sqrt n}-e^{it_2\bvar/\sqrt n})]}
\le Cn^{-\delta/2}|t_1|(1+|\mathbf{t}|)^4e^{-c_0t_2^2}.
\]
\end{itemize}
\end{prop}

Before giving the proof, let's recall some basic inequalities.
For any $x\in\R$ and any $\delta\in[0,1]$
\begin{equation}\label{eq:exp-diff}
|e^{ix}-1|=2|\sin\tfrac{x}{2}|\le 2|x/2|^\delta,
\end{equation}
and so
\begin{equation}\label{eq:exp-diff2}
\Abs{e^{ix}-ix-1}
=\Abs{ix\int_0^1(e^{isx}-1)\dd s}
\le 
C_\delta|x|^{1+\delta},
\end{equation}
\begin{equation}\label{eq:exp-diff3}
\Abs{e^{ix}-(1+ix-x^2/2)}
=
\Abs{x^2\int_0^1(s-1)(e^{isx}-1)\dd s}
\le C_\delta |x|^{2+\delta}.
\end{equation}
\begin{proof}
{\bf (a)}
We first consider the case $|\mathbf{t}|\le 2\varepsilon\sqrt n$ for small enough $\e>0$ to be determined later.
By \eqref{eq:exp-diff2} and \eqref{eq:exp-diff3},
\begin{equation}\label{eq:exp-ch1}
\Abs{\vphi(\tfrac{\mathbf{t}}{\sqrt n})-1}
\le 
C_\delta |\mathbf{t}/\sqrt n|^{1+\delta}E[|\vvar|^{1+\delta}],
\end{equation}
\begin{equation}\label{eq:exp-chf2}
\Abs{\vphi(\tfrac{\mathbf{t}}{\sqrt n})-(1-\mathbf{t}^T\Sigma \mathbf{t}/2n)}
\le 
C_\delta |\mathbf{t}/\sqrt n|^{2+\delta} E[|\vvar|^{2+\delta}].
\end{equation}
We take $\varepsilon>0$ to be small enough  such that $|\vphi(\mathbf{t}/\sqrt n)-1|<0.5$  when $|\mathbf{t}|\le2\varepsilon\sqrt n$. In this case $\log \vphi(\mathbf{t}/\sqrt n)$ is well-defined for $|\mathbf{t}|/\sqrt n\le 2\varepsilon$, and 
\begin{align*}
\Abs{\log\vphi(\tfrac{\mathbf{t}}{\sqrt n})+\mathbf{t}^T\Sigma \mathbf{t}/2n}
&=
\Abs{
\log\left(1-[1-\vphi(\tfrac{\mathbf{t}}{\sqrt n})]\right)+\mathbf{t}^T\Sigma \mathbf{t}/2n
}\\
&=
\Abs{
\sum_{k=2}^\infty
\frac{(\vphi(\tfrac{\mathbf{t}}{\sqrt n})-1)^k}{k}+\vphi(\tfrac{\mathbf{t}}{\sqrt n})-1+\mathbf{t}^T\Sigma \mathbf{t}/2n
}\\
&\stackrel{\eqref{eq:exp-chf2}}{\le}
\Abs{\vphi(\tfrac{\mathbf{t}}{\sqrt n})-1}^2\sum_{k=2}^\infty\frac{2^{2-k}}{k}
+C(\tfrac{|\mathbf{t}|}{\sqrt n})^{2+\delta}\\
&\stackrel{\eqref{eq:exp-ch1}}{\le}
C(\tfrac{|\mathbf{t}|}{\sqrt n})^{2+\delta}.
\end{align*}
Further, for $|\mathbf{t}|/\sqrt n\le 2\varepsilon$, using the inequality $|e^x-1|\le |x|e^{|x|}$, for $0\le j\le 2$,
\begin{align*}
&\Abs{\vphi(\tfrac{\mathbf{t}}{\sqrt n})^{n-j}-\lambda_0(\mathbf{t})}\\
&=
e^{-(n-j)\mathbf{t}^T\Sigma \mathbf{t}/2n}
\Abs{e^{(n-j)\left(\log\vphi(\tfrac{\mathbf{t}}{\sqrt n})+\mathbf{t}^T\Sigma \mathbf{t}/2n\right)}-1}+\Abs{e^{-(n-j)\mathbf{t}^T\Sigma \mathbf{t}/2n}-e^{-\mathbf{t}^T\Sigma \mathbf{t}/2}}\\
&\le 
C e^{-c|\mathbf{t}|^2}n(\tfrac{|\mathbf{t}|}{\sqrt n})^{2+\delta}
e^{C\varepsilon^{\delta} |\mathbf{t}|^2}+C\tfrac{|\mathbf{t}|^2}{n}e^{-c|\mathbf{t}|^2}\\
&\le 
C n^{-\delta/2} (|\mathbf{t}|+1)^{2+\delta}e^{-c|\mathbf{t}|^2},
\end{align*}
where the last inequality holds if the constant $\e>0$ is sufficiently small. This completes the proof of part (a) for $|\mathbf{t}|\le 2\varepsilon\sqrt n$.

It remains to consider the case $\varepsilon\sqrt n\le |t_2|\le\pi\sqrt n$. Since the random variable $\bvar$ has a lattice distribution with span 1, by \cite[Lemma~1, \S~2]{bo80}, when $\varepsilon'\in(0,\e)$ is small enough,  then there exists $\gamma=\gamma(\varepsilon,\varepsilon')\in(0,1)$ such that 
\begin{equation}\label{eq:e17}
|\vphi(\tfrac{\mathbf{t}}{\sqrt n})|\le 1-\gamma, \quad\forall\, |t_1|\le\varepsilon'\sqrt n,\, \varepsilon\sqrt n\le |t_2|\le \pi\sqrt n.
\end{equation}
 Hence, when $|t_1|\le\varepsilon'\sqrt n$ and $\varepsilon\sqrt n\le |t_2|\le \pi\sqrt n$, for $j=0,1,2$,
\[|\vphi(\tfrac{\mathbf{t}}{\sqrt n})^{n-j}|+\lambda_0(\mathbf{t})\le Ce^{-cn}\le Cn^{-\delta/2}e^{-ct_2^2}\le Cn^{-\delta/2}e^{-c|\mathbf{t}|^2}.\]
Therefore, we have proved that (a) holds whenever $|t_1|\le \varepsilon'\sqrt n, |t_2|\le\pi\sqrt n$.

\noindent{\bf (b)}
Note that
\begin{equation}\label{eq:rv-chf-2der}
\tfrac{\partial^2}{\partial t_2^2}\lambda_n(\mathbf{t})
=-(n-1)\vphi(\tfrac{\mathbf{t}}{\sqrt n})^{n-2}E[\bvar e^{i\mathbf{t}\cdot \vvar/\sqrt n}]^2-\vphi(\tfrac{\mathbf{t}}{\sqrt n})^{n-1}E[\bvar^2e^{i\mathbf{t}\cdot \vvar/\sqrt n}].
\end{equation}
First, for any $\mathbf{t}=(t_1,t_2)\in\R^2$, 
\begin{align}\label{eq:e33}
\Abs{E[\bvar e^{i\mathbf{t}\cdot \vvar}-i(t_1\sigma_{12}+\sigma^2_2t_2)]}
&=
\Abs{E[\bvar(e^{i\mathbf{t}\cdot \vvar}-i\mathbf{t}\cdot \vvar-1)]}\nn\\
&\stackrel{\eqref{eq:exp-diff2} }{\le}
C_\delta |\mathbf{t}|^{1+\d} E\left[\bvar|\vvar|^{1+\d}\right]  \le 
C|\mathbf{t}|^{1+\delta}.
\end{align}
Thus for any $\mathbf{t}\in\R^2$, (using $|z^2-w^2|\leq |z-w|^2+2|z-w||w|$)
\begin{equation}\label{eq:1chZ-sq-diff}
\Abs{E[\bvar e^{i\mathbf{t}\cdot \vvar}]^2+(t_1\sigma_{12}+\sigma^2_2t_2)^2}
\le  C |\mathbf{t}|^{2+\delta}(1+|\mathbf{t}|^\delta).
\end{equation}
Next, for any $\mathbf{t}\in\R^2$,
\begin{equation}\label{eq:2chZ-sq-diff}
|E[\bvar^2e^{i\mathbf{t}\cdot \vvar}-\sigma_2^2]|
=
|E[\bvar^2(e^{i\mathbf{t}\cdot \vvar}-1)]|
\stackrel{\eqref{eq:exp-diff}}{\le}
CE[\bvar^2|\mathbf{t}\cdot \vvar|^\delta]
\le C|\mathbf{t}|^{\delta}.
\end{equation}
Combining \eqref{eq:1chZ-sq-diff} and \eqref{eq:2chZ-sq-diff}, we obtain for $|\mathbf{t}| \leq 2\pi\sqrt{n}$
\begin{equation*}
\Abs{\tfrac{\partial^2}{\partial t_2^2}\lambda_n(\mathbf{t})-\tfrac{n-1}{n}\vphi(\tfrac{\mathbf{t}}{\sqrt n})^{n-2}(t_1\sigma_{12}+\sigma^2_2t_2)^2+\vphi(\tfrac{\mathbf{t}}{\sqrt n})^{n-1}\sigma_2^2}
\le
C \frac{|\mathbf{t}|^{2+\d}+|\mathbf{t}|^\d}{n^{\d/2}} \Abs{\vphi(\tfrac{\mathbf{t}}{\sqrt{n}})}^{n-2}
\end{equation*}
Furthermore, since
\begin{equation}\label{eq:e4}
\frac{\partial^2}{\partial t_2^2}\lambda_0(\mathbf{t})=(\sigma_{12}t_1+\sigma_2^2t_2)^2\lambda_0(\mathbf{t})-\sigma_2^2\lambda_0(\mathbf{t}),
\end{equation}
we have for $|\mathbf{t}|\leq  2\pi\sqrt{n}$,
\begin{align}\label{eq:e5}
\MoveEqLeft\Abs{\tfrac{\partial^2}{\partial t_2^2}(\lambda_n(\mathbf{t})-\lambda_0(\mathbf{t}))}\\
&\le 
C|\mathbf{t}|^2\Abs{\tfrac{n-1}{n}\vphi(\tfrac{\mathbf{t}}{\sqrt n})^{n-2}-\lambda_0(\mathbf{t})}
+C\Abs{\vphi(\tfrac{\mathbf{t}}{\sqrt n})^{n-1}-\lambda_0(\mathbf{t})}+C\tfrac{|\mathbf{t}|^{2+\d}+|\mathbf{t}|^\d}{n^{\delta/2}}\Abs{\vphi(\tfrac{\mathbf{t}}{\sqrt n})}^{n-2}.\nn
\end{align}
Note that (a) implies that
\begin{equation}\label{eq:e34}
|\vphi(\tfrac{\mathbf{t}}{\sqrt n})^{n-2}|\le Ce^{-c|\mathbf{t}|^2} \quad\mbox{when } |t_1|\le \varepsilon\sqrt n \mbox{ and } |t_2|\le\pi\sqrt n.
\end{equation}
Statement (b) now follows from (a) and \eqref{eq:e5}.

\noindent{\bf (c)}
In what follows, for $t=(t_1,t_2)$, we let $\bar{\mathbf{t}}=(\bar t_1,\bar t_2):=\mathbf{t}/\sqrt n$ denote the rescaled vector. Set $H_j(\mathbf{t})=\vphi(\bar{\mathbf{t}})^{n-j}-\lambda_0(\mathbf{t}), j=0,1,2$, and let $K_1(\mathbf{t}):=E[\bvar^2e^{i\bar{\mathbf{t}}\cdot \vvar}], K_2(\mathbf{t}):=E[\bvar e^{i\bar{\mathbf{t}}\cdot \vvar}],  K_3(\mathbf{t}):=-(\sigma_{12}t_1+\sigma_2^2 t_2)^2$.  
We define the functions $\tilde\Lambda(\mathbf{t})=\Lambda(0,t_2), \tilde\Lambda_0(\mathbf{t})=\lambda_0(0,t_2), \tilde H_j(\mathbf{t})=H_j(0,t_2)$ and $\tilde K_i(\mathbf{t})=K_i(0,t_2)$, $0\le j\le 2, 1\le i\le 3$. 
Our goal is to obtain a bound  for
\[
\Lambda(\mathbf{t})-\tilde\Lambda(\mathbf{t})+\vphi(0,\bar t_2)^{n-1}(K_1-\tilde K_1).
\]
By \eqref{eq:rv-chf-2der} and \eqref{eq:e4},
\[
\Lambda(\mathbf{t})=-(n-1)H_2K_2^2-H_1K_1-\lambda_0[(n-1)K_2^2-K_3+(K_1-\sigma_2^2)]
\]
Setting $\Delta_j:=H_j-\tilde H_j, 0\le j\le 2$, we have (Note that $\vphi(0,\bar t_2)^{n-1}=\tilde H_1+\tilde\Lambda_0$.)
\begin{align}\label{eq:e7}
\MoveEqLeft\Lambda(\mathbf{t})-\tilde\Lambda(\mathbf{t})+\vphi(0,\bar t_2)^{n-1}(K_1-\tilde K_1)\nn\\
&=
-[(n-1)\Delta_2K_2^2+\Delta_1K_1]-(\lambda_0-\tilde\Lambda_0)[(n-1)K_2^2-K_3+(K_1-\sigma_2^2)]\nn\\
&-(n-1)\tilde H_2(K_2^2-\tilde K_2^2)-\tilde\Lambda_0[(n-1)(K_2^2-\tilde K_2^2)-(K_3-\tilde K_3)]\nn\\
&:=I_1+I_2+I_3+I_4.
\end{align}
We will estimate the four terms in the following steps.
\begin{enumerate}[Step 1.]
\item
To estimate $I_1$, we will first show that for $|t_1|\le \varepsilon\sqrt n, |t_2|\le\pi\sqrt n$,
\begin{equation}\label{eq:e6}
|\Delta_j|\le Cn^{-\delta/2}|t_1|e^{-ct_2^2}, 
\qquad 0\le j\le 2.
\end{equation} 
For simplicity we only provide proof for the case $j=0$.
By (a) and \eqref{eq:e34},
\begin{align*}
|\tfrac{\partial}{\partial t_1}H_0|
&=|\tfrac{\partial}{\partial t_1}(\vphi(\bar{\mathbf{t}})^n-\lambda_0(\mathbf{t}))|\\
&=|\sqrt n\vphi(\bar{\mathbf{t}})^{n-1}E[i\avar(e^{i\bar{\mathbf{t}}\cdot \vvar}-i\bar{\mathbf{t}}\cdot \vvar-1)]+(\lambda_0-\vphi(\bar{\mathbf{t}})^{n-1})(t_1\sigma_1^2+t_2\sigma_{12})|\\
&\stackrel{\eqref{eq:exp-diff2}}{\le}
Cn^{-\delta/2}e^{-c|\mathbf{t}|^2}.
\end{align*}
Thus 
$|\Delta_0(\mathbf{t})|=\Abs{\int_0^{t_1}\frac{\partial}{\partial t_1}H_0(s,t_2)\dd s}
\le 
Cn^{-\delta/2}|t_1|e^{-ct_2^2}$. Display \eqref{eq:e6} is proved for $j=0$.
The proofs for $j=1,2$ are similar. Further, by \eqref{eq:e33} and \eqref{eq:2chZ-sq-diff},  we have$|K_2|\le Cn^{-1/2}|\mathbf{t}|$ and $|K_1|\le C$ when $|\bar t_1|\le\varepsilon, |\bar t_2|\le\pi$. Hence, the term $I_1$ defined in \eqref{eq:e7} has bound
\[
|I_1|\le Cn^{-\delta/2}|t_1|(1+|\mathbf{t}|)^2e^{-ct_2^2}.
\]
\item 
To estimate $I_2$, noting that $|\lambda_0-\tilde\Lambda_0|\le C|t_1||\mathbf{t}|e^{-ct_2^2}$, it suffices to show that  
\begin{equation}\label{eq:e8}
|(n-1)K_2^2-K_3+(K_1-\sigma_2^2)|\le Cn^{-\delta/2}(1+|\mathbf{t}|)^3.
\end{equation}
By \eqref{eq:1chZ-sq-diff} and \eqref{eq:2chZ-sq-diff}, when $|\mathbf{t}|\le 2\pi\sqrt n$, we have
$|nK_2^2-K_3|\le Cn^{-\delta/2}|\mathbf{t}|^{2+\delta}$ and $|K_1-\sigma_2^2|\le Cn^{-\delta/2}|\mathbf{t}|^\delta$.
Thus \eqref{eq:e8} is obtained and we can conclude that  for $|\mathbf{t}|\le2\pi\sqrt n$,
\[
|I_2|\le 
Cn^{-\delta/2}|t_1|(1+|\mathbf{t}|)^4e^{-ct_2^2}.
\]
\item 
To estimate $I_4$, it suffices to prove that for $|\mathbf{t}|\le2\pi\sqrt n$,
\begin{equation}\label{eq:e9}
|(n-1)(K_2^2-\tilde K_2^2)-(K_3-\tilde K_3)|\le Cn^{-\delta/2}|t_1||\mathbf{t}|^{1+\delta}.
\end{equation}
Indeed, by \eqref{eq:exp-diff2},
$|K_2-(\tilde K_2+i\sigma_{12}\bar t_1)|=|E[\bvar e^{i\bar t_2\bvar}(e^{i\bar t_1\avar}-i\bar t_1 \avar-1)]|\le C|\bar t_1|^{1+\delta}$. Further, by \eqref{eq:e33}, we have $|K_2|+|\tilde K_2|\le C|\bar{\mathbf{t}}|$  when $|\bar{\mathbf{t}}|\le 2\pi$. Hence $|K_2^2-(\tilde K_2+i\sigma_{12}\bar t_1)^2|\le Cn^{-(2+\delta)/2}|t_1||\mathbf{t}|^{1+\delta}$. On the other hand, 
\[
|n((\tilde K_2+i\sigma_{12}\bar t_1)^2-\tilde K_2^2)-(K_3-\tilde K_3)|
=\Abs{2i\sqrt n\sigma_{12}t_1 E[\bvar(e^{i\bar t_2 \bvar}-i\bar t_2\bvar-1)]}
\le Cn^{-\delta/2}|t_1||\mathbf{t}|^{1+\delta}.
\]
Thus we conclude that when $|\mathbf{t}|\le 2\pi\sqrt n$,
\[
|n(K_2^2-\tilde K_2^2)-(K_3-\tilde K_3)|\le Cn^{-\delta/2}|t_1||\mathbf{t}|^{1+\delta}.
\]
Noticing that $|K_3-\tilde K_3|\le C|t_1||\mathbf{t}|$, we get
\begin{equation}\label{eq:e10}
n|K_2^2-\tilde K_2^2|
\le C|t_1||\mathbf{t}| \qquad\text{ when }|\mathbf{t}|\le 2\pi\sqrt n.
\end{equation}
Display \eqref{eq:e9} then follows, and we obtain for $|\mathbf{t}|\le2\pi\sqrt n$,
\[
|I_4|\le 
Cn^{-\delta/2}|t_1||\mathbf{t}|^{1+\delta}e^{-ct_2^2}.
\]
\item 
Finally, by (a), we have $|\tilde H_2|\le Cn^{-\delta/2}e^{-ct_2^2}$. This inequality, together with \eqref{eq:e10}, yields
\[
|I_3|\le  Cn^{-\delta/2}|t_1||\mathbf{t}|e^{-ct_2^2} \qquad\text{when }|t_1|\le\varepsilon\sqrt n \text{ and }|t_2|\le\pi\sqrt n.
\]
\end{enumerate}
Our proof is complete.
\end{proof}

\subsection{Proof of Proposition~\ref{thm:LLT}}\label{subsec:pf-llt}

When $B$ is a continuous random variable, the proof of Proposition~\ref{thm:LLT} can be found in \cite{shev17} or \cite{bcg11}. For our case where $B$ is a discrete random variable,  we include the proof as follows for the purpose of completeness, since it is rather elementary.

\begin{proof}[Proof of Proposition~\ref{thm:LLT}]
First, we will express the right-hand side of the equality in terms of the characteristic function.
We let $\tilde\lambda_0(t)=\exp(-\sigma_2^2t^2/2)$ and let $\tilde\lambda_n(t)$, $t\in\R$,  denotes the characteristic functions of $\bsum_n/\sqrt n$. Then for any $y\in\Z$,
\begin{equation}\label{eq:e35}
1_{\bsum_n=y+n\rho}=\frac{1}{2\pi}\int_{-\pi}^\pi e^{it (\bsum_n-n\rho)} e^{-it y}\dd t
=\frac{1}{2\pi\sqrt n}\int_{-\pi\sqrt n}^{\pi\sqrt n} e^{it \bsum_n/\sqrt n} e^{-it y_n}\dd t
\end{equation}
and so
\begin{align*}
P(\bsum_n/\sqrt n=y_n)=
\frac{1}{2\pi\sqrt n}\int_{-\pi\sqrt n}^{\pi\sqrt n}\tilde\lambda_n(t) e^{-it y_n}\dd t.
\end{align*}
Using integration by parts, we get 
\[
y_n^2 P(\bsum_n/\sqrt n=y_n)
=
\frac{-1}{2\pi\sqrt n}\int_{-\pi\sqrt n}^{\pi\sqrt n}
\tilde\lambda_n''(t)e^{-it y_n}\dd t
\]
and
\[
\tfrac{y_n^2}{\sqrt{2\pi \sigma_2^2}}e^{-y_n^2/2\sigma_2^2}
=
\frac{-1}{2\pi}\int_{-\infty}^\infty \tilde\lambda_0''(t)e^{-ity_n}\dd t.
\]
Thus
\begin{align*}
\MoveEqLeft(1+y_n^2)\Abs{\sqrt n P(\bsum_n/\sqrt n=y_n)-\tfrac{1}{\sqrt{2\pi \sigma_2^2}}e^{-y_n^2/2\sigma_2^2}}\nonumber\\
&=
\frac{1}{2\pi}\Abs{
\int_{-\pi\sqrt n}^{\pi\sqrt n}(\tilde\lambda_n-\tilde\lambda_n'')e^{-it y_n}\dd t-
\int_{-\infty}^\infty (\tilde\lambda_0-\tilde\lambda_0'')e^{-ity_n}\dd t}\nonumber\\
&\le 
\frac{1}{2\pi}\int_{-\pi\sqrt n}^{\pi\sqrt n}
\Abs{\tilde\lambda_n-\tilde\lambda_n''-\tilde\lambda_0+\tilde\lambda_0''}\dd t
+\int_{|t|>\pi\sqrt n}|\tilde\lambda_0-\tilde\lambda_0''|\dd t.
\end{align*}
Note that $\int_{|t|>\pi\sqrt n}|\tilde\lambda_0-\tilde\lambda_0''|\dd t\le Ce^{-cn}$. 
On the other hand, by Proposition~\ref{prop:chf-for-LLT}(a)(b),
\begin{align*}
\int_{-\pi\sqrt n}^{\pi\sqrt n}
\Abs{\tilde\lambda_n-\tilde\lambda_n''-\tilde\lambda_0+\tilde\lambda_0''}\dd t
\le 
\int_{-\pi\sqrt n}^{\pi\sqrt n}Cn^{-\delta/2}e^{-ct^2}\dd t
\le Cn^{-\delta/2}.
\end{align*}
The proposition follows.
\end{proof}

\subsection{Proof of Theorem~\ref{thm:semi-local-LLT}}\label{subsec:pf-semi-llt}

The proof relies on the expression (cf. \eqref{eq:kmgv-dist} and \eqref{eq:e14}) of the Kolmogorov distance in terms of characteristic functions, where a probability measure $v_J$ is introduced to make the distribution functions smooth and to truncate their characteristic functions. To be specific, define the measure $v_\polya (\dd x):=\tfrac{1-\cos(\polya x)}{\pi \polya x^2}\dd x$ on $\R$, where $\polya>0$ is a constant to be determined.
Note that its characteristic function $\hat v_\polya (x)=(1-\tfrac{|x|}{\polya})_+$ is supported on $[-\polya,\polya]$.
\begin{proof}
In what follows, for any measure (or distribution function) $\mu$, we denote its  characteristic function by $\hat\mu$. Recall that the characteristic functions of $\tfrac{\vsum_n}{\sqrt n}$, $\mathcal N$ are denoted by
$\lambda_n(\mathbf{t})$ and $\lambda_0(\mathbf{t})$, $\mathbf{t}\in\R^2$. Also, for simplicity we will suppress the subscript $\Sigma$  and write $\psi_\Sigma$ simply as $\psi$.

\begin{enumerate}[{\it Step 1.}]
\item First, we will express the left-side of Theorem~\ref{thm:semi-local-LLT} in terms of measures with compactly supported characteristic functions, i.e.\ \eqref{eq:e13}.
For any fixed  $y\in\Z$, let $F_n(x,y_n):=P(\asum_n/\sqrt n\le x, \bsum_n/\sqrt n=y_n)$ and denote the corresponding conditional distribution functions by $\bar F_n(x):=\tfrac{F_n(x,y_n)}{F_n(\infty, y_n)}$, $\bar \psi_n(x):=\tfrac{\psi(x,y_n)}{\psi(\infty, y_n)}$.
Of course, since the case $F_n(\infty,y_n)=0$ follows immediately from Proposition~\ref{thm:LLT}, we only consider the non-trivial case when $F_n(\infty,y_n)>0$, so that $\bar F_n$ is well-defined.
Then 
\begin{align*}
\MoveEqLeft(1+y_n^2)\Abs{\sqrt n F_n(x,y_n)-\psi(x,y_n)}\\
&=
(1+y_n^2)\Abs{(\bar F_n(x)-\bar\psi_n(x))\psi(\infty,y_n)+(\sqrt n F_n(\infty, y_n)-\psi(\infty,y_n))\bar F_n(x)}\\
&\le 
(1+y_n^2)\psi(\infty,y_n)\Abs{\bar F_n(x)-\bar\psi_n(x)}+Cn^{-\delta/2},
\end{align*}
where in the last inequality we used Proposition~\ref{thm:LLT}.
Further,  let $\bar F_n^\polya $ (and $\bar \psi_n^\polya$) be the convolution of $\bar F_n$ (and $\bar\psi_n$, resp.) and the measure $v_\polya$. Then, by \cite[Lemma~1, XVI.3]{fe71},
\begin{equation}\label{eq:kmgv-dist}
\sup_x\Abs{\bar F_n(x)-\bar\psi_n(x)}
\le 
2\sup_x\Abs{\bar F_n^\polya (x)-\bar \psi_n^\polya (x)}+\frac{24}{\pi\polya}\sup_x\Abs{\frac{\partial}{\partial x}\bar\psi_n(x)}.
\end{equation}
From now on we take $\polya=\varepsilon\sqrt n$, where $\varepsilon$ is the constant in Proposition~\ref{prop:chf-for-LLT}. Collecting the above inequalities we get 
\begin{equation}\label{eq:e13}
\sup_{x\in\R,y\in\Z}(1+y_n^2)\Abs{\sqrt n F_n(x,y_n)-\psi(x,y_n)}
\le 
C\sup_{x\in\R,y\in\Z}(1+y_n^2)\psi(\infty,y_n)\Abs{\bar F_n^\polya (x)-\bar \psi_n^\polya (x)}+Cn^{-\delta/2}.
\end{equation}

\item 
Let 
\[
\Delta_n^\polya (x):=\bar F_n^\polya (x)-\bar \psi_n^\polya (x).
\]
 Our second step is to write $\Delta_n^\polya $ in terms of characteristic functions, cf \eqref{eq:e15}.
By Fourier's inversion formula for distribution functions \cite[(3.11), XV.4]{fe71}, for any $x>a$,
\begin{align}\label{eq:e14}
\bar F_n^\polya (x)-\bar F_n^\polya (a)
&=
\frac{1}{2\pi}\int_{-\polya}^\polya 
\frac{e^{-it_1x}-e^{-it_1 a}}{it_1}\hat{\bar F}_n^\polya (t_1)\dd t_1,\\
\bar \psi_n^\polya (x)-\bar \psi_n^\polya (a)
&=
\frac{1}{2\pi}\int_{-\polya}^\polya 
\frac{e^{-it_1x}-e^{-it_1 a}}{it_1}\hat{\bar \psi}_n^\polya (t_1)\dd t_1.\nn
\end{align}
Note that (let $\mathbf{t}:=(t_1,t_2)$)
\begin{align}\label{eq:e20}
\hat{\bar F}_n^\polya (t_1)=\hat{\bar F}_n(t_1)\hat v_\polya (t_1)
&=\frac{\hat v_\polya (t_1)}{F_n(\infty,y_n)}
E[e^{i\asum_n t_1/\sqrt n}1_{\bsum_n/\sqrt n=y_n}]\nn\\
&\stackrel{\eqref{eq:e35}}{=}
\frac{\hat v_\polya (t_1)}{2\pi\sqrt n F_n(\infty,y_n)}
\int_{-\pi\sqrt n}^{\pi\sqrt n}\lambda_n(\mathbf{t})e^{-it_2 y_n}\dd t_2.
\end{align}
On the other hand,
\begin{align}\label{eq:e21}
\hat{\bar \psi}_n^\polya (t_1)=\hat{\bar \psi}_n (t_1)\hat v_\polya (t_1)
=
\frac{\hat v_\polya (t_1)}{2\pi\psi(\infty,y_n)}\int_{-\infty}^\infty
\lambda_0(\mathbf{t})e^{-it_2 y_n}\dd t_2.
\end{align}
These equalities, together with those in \eqref{eq:e14}, yield

\begin{align}\label{eq:e30}
\MoveEqLeft\sqrt n F_n(\infty, y_n)(\bar F_n^\polya (x)-\bar F_n^\polya (a))-\psi(\infty, y_n)(\bar \psi_n^\polya (x)-\bar \psi_n^\polya (a))\\
&=
\int_{|t_1|\le\polya,t_2\in\R}\frac{\hat v_\polya (t_1)}{(2\pi)^2}\cdot \frac{e^{-it_1 a}-e^{-it_1x}}{it_1}
e^{-it_2y_n}\left(
\lambda_n(\mathbf{t})1_{|t_2|\le \pi\sqrt n}
-\lambda_0(\mathbf{t})
\right)\dd\mathbf{t}. \nonumber
\end{align}

Further, integration by parts in \eqref{eq:e20} and \eqref{eq:e21} gives 
\[
y_n^2\hat{\bar F}_n^\polya (t_1)
=
-\frac{\hat v_\polya (t_1)}{2\pi\sqrt n F_n(\infty,y_n)}
\int_{-\pi\sqrt n}^{\pi\sqrt n}
e^{-it_2 y_n}\tfrac{\partial^2}{\partial t_2^2}\lambda_n(\mathbf{t})\dd t_2,
\]
\[
y_n^2\hat{\bar \psi}_n^\polya (t_1)
=
-\frac{\hat v_\polya (t_1)}{2\pi\psi(\infty,y_n)}\int_{-\infty}^\infty
e^{-it_2 y_n}\tfrac{\partial^2}{\partial t_2^2}\lambda_0(\mathbf{t})\dd t_2.
\]
Similar to \eqref{eq:e30}, we then have
\begin{align}\label{eq:e31}
\MoveEqLeft y_n^2\left[\sqrt n F_n(\infty, y_n)(\bar F_n^\polya (x)-\bar F_n^\polya (a))-\psi(\infty, y_n)(\bar \psi_n^\polya (x)-\bar \psi_n^\polya (a))\right]\\
&=
\int_{|t_1|\le\polya,t_2\in\R}\frac{\hat v_\polya (t_1)}{(2\pi)^2}\cdot \frac{e^{-it_1 a}-e^{-it_1x}}{it_1}
e^{-it_2y_n}\left(
\tfrac{\partial^2}{\partial t_2^2}\lambda_n(\mathbf{t})1_{|t_2|\le \pi\sqrt n}
-\tfrac{\partial^2}{\partial t_2^2}\lambda_0(\mathbf{t})
\right)\dd\mathbf{t}.\nonumber
\end{align}
Combining \eqref{eq:e30} and \eqref{eq:e31}, we get for any $x>a$, 
\begin{align}\label{eq:e15}
&(1+y_n^2)\psi(\infty,y_n)(\Delta_n^\polya (x)-\Delta_n^\polya (a))\\
&=
(1+y_n^2)[\psi(\infty,y_n)-\sqrt n F_n(\infty,y_n)](\bar F_n^\polya (x)-\bar F_n^\polya (a))\nn
\\
&+\int_{|t_1|\le\polya,t_2\in\R}G_{n,\polya}(t_1)
e^{-it_2y_n}\left[
(\lambda_n(\mathbf{t})-\tfrac{\partial^2}{\partial t_2^2}\lambda_n(\mathbf{t}))1_{|t_2|\le \pi\sqrt n}
-\lambda_0(\mathbf{t})+\tfrac{\partial^2}{\partial t_2^2}\lambda_0(\mathbf{t})
\right]\dd\mathbf{t},\nn
\end{align}
where
\begin{equation}\label{eq:e11}
G_{n,\polya}(t_1)=G_{n,\polya}(t_1,x,a):=
\frac{\hat v_\polya (t_1)}{(2\pi)^2}\cdot \frac{e^{-it_1 a}-e^{-it_1x}}{it_1}.
\end{equation}

\item 
Our next goal is to bound \eqref{eq:e15} by $Cn^{-\delta/2}$.  Set
\[
U(\mathbf{t}):=(\lambda_n-\lambda_0)-\tfrac{\partial^2}{\partial t_2^2}(\lambda_n-\lambda_0).
\]
Note that by \eqref{eq:e15} and Proposition~\ref{thm:LLT}, we have for $x>a$,
\begin{align}\label{eq:e32}
\MoveEqLeft(1+y_n^2)\psi(\infty,y_n)\Abs{\Delta_n^\polya (x)-\Delta_n^\polya (a)}\nn\\
&\le Cn^{-\delta/2}+
\Abs{\int_{|t_1|\le\polya,|t_2|>\pi\sqrt n}G_{n,\polya}(t_1)e^{-it_2y_n}(\lambda_0-\tfrac{\partial^2}{\partial t_2^2}\lambda_0)\dd\mathbf{t}}\nn\\
&\qquad+\Abs{\int_{|t_1|\le\polya, |t_2|\le \pi\sqrt n}G_{n,\polya}(t_1)e^{-it_2y_n}U(\mathbf{t})\dd\mathbf{t}}\nn\\
&=:Cn^{-\delta/2}+ I_5+I_6.
\end{align}
We  start with $I_6$.
Recall $\polya=\varepsilon\sqrt n$ and for any $K>0$ let $\mathcal G_n(K)$ denote the set of ``good" functions $f:\R^2\to\C$ such that 
\[
\sup_{x,y,a} \Abs{\int_{|t_1|\le\polya,|t_2|\le\pi\sqrt n}G_{n,\polya}(t_1,x,a)e^{-it_2y_n}f(\mathbf{t})\dd\mathbf{t}}
\le K n^{-\delta/2}.
\]
We will show that 
\begin{equation}\label{eq:e16}
U(\mathbf{t})\in\mathcal G_n(C).
\end{equation}
Notice that every $f:\R^2\to\C$ that satisfies $|f(\mathbf{t})|\le Cn^{-\delta/2}|t_1|e^{-c|\mathbf{t}|^2}$ for $|t_1|\le\varepsilon\sqrt n, |t_2|\le\pi\sqrt n$ is in $\mathcal G_n(C)$. Set 
\begin{equation}\label{eq:e19}
R(\mathbf{t}):=\vphi(0,\tfrac{t_2}{\sqrt n})^{n-1}E[\bvar^2(e^{i\mathbf{t}\cdot \vvar/\sqrt n}-e^{it_2\bvar/\sqrt n})].
\end{equation}
Then, letting $c_0$ be the same as in Proposition~\ref{prop:chf-for-LLT}(c), 
\[
U(\mathbf{t})=
e^{-c_0t_1^2}(U-U(0,t_2)+R)+(1-e^{-c_0t_1^2})U
+e^{-c_0t_1^2}U(0,t_2)-e^{-c_0t_1^2}R.
\]
We will show that $U\in\mathcal G_n(C)$ by showing that all the four terms on the right above are in $\mathcal G_n(C)$. Note that the constant $C$ may differ for each of these four terms.
When $|t_1|\le\varepsilon\sqrt n$ and $|t_2|\le\pi\sqrt n$, by \eqref{eq:e6}, $|\lambda_n(\mathbf{t})-\lambda_0(\mathbf{t})-[\lambda_n(0,t_2)-\lambda_0(0,t_2)]|\le Cn^{-\delta/2}|t_1|e^{-ct_2^2}$. This inequality and Proposition~\ref{prop:chf-for-LLT}(c) yield
$e^{-c_0t_1^2}|U-U(0,t_2)+R|\le Cn^{-\delta/2}|t_1|e^{-c|\mathbf{t}|^2}$. Hence  there exists a constant $C_1$ such that 
$e^{-c_0t_1^2}(U-U(0,t_2)+R)\in\mathcal G_n(C_1)$. Also, using $1-e^{-c_0t_1^2}\le Ct_1^2$ and Proposition~\ref{prop:chf-for-LLT}(a)(b), we have $(1-e^{-c_0t_1^2})U(\mathbf{t})\in\mathcal G_n(C_2)$  for some constant $C_2$. Further, 
\begin{align}\label{eq:e36}
&\Abs{\int_{|t_1|\le\polya,|t_2|\le\pi\sqrt n}G_{n,\polya}(t_1)e^{-it_2y_n}e^{-c_0t_1^2}U(0,t_2)\dd\mathbf{t}}\\
&\le 
\Abs{\int_{|t_1|\le\polya}G_{n,\polya}(t_1)e^{-c_0t_1^2}\dd t_1}
\Abs{\int_{|t_2|\le \pi\sqrt n}U(0,t_2)e^{-it_2y_n}\dd t_2}.\nn
\end{align}
By the inversion formula, for $x>a$, the first integral $\int_{|t_1|\le\polya}G_{n,\polya}(t_1)e^{-c_0t_1^2}\dd t_1=\mu_J(a,x)/2\pi<1/2\pi$, where $\mu_\polya $ denotes the probability measure of  $v_\polya *\mathcal Z_{2c_0}$ and $\mathcal Z_{2c_0}$ denotes the normal distribution with mean 0 and variance $2c_0$. On the other hand, by Proposition~\ref{prop:chf-for-LLT}(a)(b), we have $|U(0,t_2)|\le Cn^{-\delta/2}e^{-ct_2^2}$ for $|t_2|\le\pi\sqrt n$, which implies $\Abs{\int_{|t_2|\le \pi\sqrt n}U(0,t_2)e^{-it_2y_n}\dd t_2}\le Cn^{-\delta/2}$. Hence the integral in \eqref{eq:e36} is bounded by $Cn^{-\delta/2}$ and so $e^{-c_0t_1^2}U(0,t_2)\in\mathcal G_n(C_3)$  for some constant $C_3$. 

To prove $U(\mathbf{t})\in\mathcal G_n(C)$ it remains to show that $e^{-c_0t_1^2}R\in\mathcal G_n(C_4)$ for some constant $C_4$. Indeed, by the fact that $\hat v_\polya $ is supported on $[-\polya,\polya]$ and Fubini's theorem, (Recall the definition of $G_{n,\polya}$ at \eqref{eq:e11}.)
\begin{align}\label{eq:e12}
&\int_{|t_1|\le\polya,|t_2|\le\pi\sqrt n}G_{n,\polya}(t_1)e^{-it_2y_n}e^{-c_0t_1^2}R(\mathbf{t})\dd\mathbf{t}\\
&=
E\left[\bvar^2
\int_{-\infty}^\infty G_{n,\polya}(t_1)e^{-c_0t_1^2}(e^{it_1\avar/\sqrt n}-1)\dd t_1
\int_{|t_2|\le \pi\sqrt n}\vphi(0,\tfrac{t_2}{\sqrt n})^{n-1}e^{-it_2y_n}e^{it_2\bvar/\sqrt n}\dd t_2
\right].\nn
\end{align}
 By the inversion formula for distribution functions, 
\begin{align*}
 \int_{-\infty}^\infty G_{n,\polya}(t_1)e^{-c_0t_1^2}(e^{it_1\avar/\sqrt n}-1)\dd t_1
 &=
C[\mu_\polya (x,x+\tfrac{\avar}{\sqrt n})- \mu_\polya (a,a+\tfrac{\avar}{\sqrt n})]1_{\avar\ge 0}\\
&+C [\mu_\polya (a+\tfrac{\avar}{\sqrt n},a)-\mu_\polya (x+\tfrac{\avar}{\sqrt n},x)]1_{\avar< 0}.
 \end{align*} 
Since $\mu_\polya $ has (by the inversion formula) bounded density, for any $x\in\R$,
\[
\mu_\polya (x,x+\tfrac{\avar}{\sqrt n})1_{\avar\ge 0}+\mu_\polya (x+\tfrac{\avar}{\sqrt n},x)1_{\avar<0}
\le C|\tfrac{\avar}{\sqrt n}|\wedge 1\le C|\tfrac{\avar}{\sqrt n}|^\delta.
\]
Also, by \eqref{eq:e34}, the second integral on the right side of \eqref{eq:e12} is bounded in absolute value  by $\int_{t_2\in\R}|\vphi(0,\tfrac{t_2}{\sqrt n})|^{n-1}\dd t_2<C$.
Then, by \eqref{eq:e12} we have
\[
\Abs{\int_{|t_1|\le\polya,|t_2|\le\pi\sqrt n}G_{n,\polya}(t_1)e^{-it_2y_n}e^{-c_0t_1^2}R(\mathbf{t})\dd \mathbf{t}}
\le 
CE[\bvar^2|\tfrac{\avar}{\sqrt n}|^\delta]\le Cn^{-\delta/2}.
\]
So $e^{-c_0t_1^2}R\in\mathcal G_n(C_4)$ for some constant $C_4>0$ and \eqref{eq:e16} is proved. Therefore $I_6\le Cn^{-\delta/2}$.
\item
To estimate $I_5$ in \eqref{eq:e32}, recall that by \eqref{eq:e4}, $|\tfrac{\partial^2}{\partial t_2^2}\lambda_0(\mathbf{t})-(\sigma_2^4t_2^2-\sigma_2^2)\lambda_0(\mathbf{t})|
\le C|t_1||\mathbf{t}|\lambda_0(\mathbf{t})
\le C|t_1|e^{-c|\mathbf{t}|^2}$.
Thus
\begin{align*}
&\Abs{\int_{|t_1|\le\polya,|t_2|>\pi\sqrt n}G_{n,\polya}(t_1)e^{-it_2y_n}[\tfrac{\partial^2}{\partial t_2^2}\lambda_0(\mathbf{t})-(\sigma_2^4t_2^2-\sigma_2^2)\lambda_0(\mathbf{t})]\dd\mathbf{t}}\\
&\le 
C\int_{|t_1|\le\polya,|t_2|>\pi\sqrt n}\frac{1}{|t_1|}|t_1|e^{-c|\mathbf{t}|^2}\dd\mathbf{t}
\le
Ce^{-cn}.
\end{align*}
On the other hand, recalling that $\mathcal N=(\mathcal N_1,\mathcal N_2)$ is the limiting normal distribution, we have $\lambda_0(\mathbf{t})=E[e^{it_1\mathcal N_1+it_2\mathcal N_2}]$. By Fubini's theorem,
\begin{align*}
&\int_{|t_1|\le\polya,|t_2|>\pi\sqrt n}G_{n,\polya}(t_1)e^{-it_2y_n}(\sigma_2^4t_2^2-\sigma_2^2-1)\lambda_0\dd\mathbf{t}\\
&=
E\left[\int_{|t_2|>\pi\sqrt n}e^{it_2(\mathcal N_2-y_n)}(\sigma_2^4t_2^2-\sigma_2^2-1)\dd t_2
\int_{|t_1|\le\polya}\frac{e^{-it_1 (a-\mathcal N_1)}-e^{-it_1(x-\mathcal N_1)}}{(2\pi)^2it_1}\hat v_\polya (t_1)\dd t_1
\right].
\end{align*}
Note that by Fourier's inversion formula (and the fact that $\hat v_\polya $ is supported on $[-\polya,\polya]$), 
\[
f(\mathcal N_1):=\frac{1}{2\pi}\int_{|t_1|\le\polya}\frac{e^{-it_1 (a-\mathcal N_1)}-e^{-it_1(x-\mathcal N_1)}}{it_1}\hat v_\polya (t_1)\dd t_1
=v_\polya (a-\mathcal N_1,x-\mathcal N_1).
\]
Thus $|f|\le 1$. Also note that conditioning on $\mathcal N_1$, the variable $\mathcal N_2$ has a normal distribution with mean $\sigma_{12}\mathcal N_1/\sigma_1^2$ and variance $\sigma_2^2-\tfrac{\sigma_{12}^2}{\sigma_1^2}$. Hence
\begin{align*}
&\Abs{\int_{|t_1|\le\polya,|t_2|>\pi\sqrt n}G_{n,\polya}(t_1)e^{-it_2y_n}(\sigma_2^4t_2^2-\sigma_2^2-1)\lambda_0\dd\mathbf{t}}\\
&= 
\frac{1}{2\pi}\Abs{E\left[\int_{|t_2|>\pi\sqrt n}e^{it_2(\mathcal N_2-y_n)}(\sigma_2^4t_2^2-\sigma_2^2-1)f(\mathcal N_1)\dd t_2\right]}
\\
&=
\frac{1}{2\pi}\Abs{
\int_{|t_2|>\pi\sqrt n}
(\sigma_2^4t_2^2-\sigma_2^2-1)E\left[\exp\left(i(\tfrac{\sigma_{12}}{\sigma_1^2}\mathcal N_1-y_n)t_2-(\sigma_2^2-\tfrac{\sigma_{12}^2}{\sigma_1^2})\tfrac{t_2^2}{2}\right)f(\mathcal N_1)\right]\dd t_2
}\\
&\le 
\int_{|t_2|>\pi\sqrt n}Ce^{-ct_2^2}\dd t_2
\le Ce^{-cn}.
\end{align*}
Therefore,  $I_5\le Ce^{-cn}$.

\item Finally, plugging the bounds $I_5\le Ce^{-cn}$ and $I_6\le Cn^{-\delta/2}$ into \eqref{eq:e32} we obtain
\[
\sup_{x\in\R, y\in\Z}(1+y_n^2)\psi(\infty,y_n)\Abs{\Delta_n^\polya (x)-\Delta_n^\polya (a)}\le Cn^{-\delta/2}.
\]
 Since the right hand side is uniform for all  $a$, we simply have
\[
\sup_{x\in\R, y\in\Z}(1+y_n^2)\psi(\infty,y_n)|\Delta_n^\polya (x)|\le Cn^{-\delta/2}.
\]
This, together with \eqref{eq:e13}, yields
\[
\sup_{x\in\R,y\in\Z}(1+y_n^2)\Abs{\sqrt n F_n(x,y_n)-\psi(x,y_n)}\le Cn^{-\delta/2}.
\]
\end{enumerate}
Our proof of Theorem~\ref{thm:semi-local-LLT} is complete.
\end{proof}

\section{Proof of the Regenerative CLT rates}\label{sec:regen}
In this section we will use the semi-local Berry Esseen estimates from Theorem \ref{thm:semi-local-LLT} in the previous section to give the proof of our main result (Theorem \ref{th:BEregCLT}). 
To more easily adapt to the i.i.d.\ setting of Theorem \ref{thm:semi-local-LLT}, we first prove the statement of Theorem \ref{th:BEregCLT} under the measure $\bP$ (that is, conditioned on a regeneration at time zero). Then, at the end of the section we show how to obtain the same results taking into account that the process is different prior to the first regeneration time.

\subsection{Proof of Theorem \ref{th:BEregCLT} \texorpdfstring{under the measure $\bP$}{conditioned on a regeneration at time zero}}

In this subsection, our aim is to prove the following Proposition which is the analog of Theorem \ref{th:BEregCLT} under the measure $\bP$. 
\begin{prop}\label{pr:BEregbP}
Let $X_n= \sum_{i=1}^n \xi_i$ be a regenerative process 
with regeneration times $\{\tau_k\}_{k\geq 1}$.
 Assume for some $\d \in (0,1]$ that 
\[
 \bE[\tau_1^{2+\d}]< \infty \quad \text{and} \quad \bE\left[ \left( \sum_{i=1}^{\tau_1} |\xi_i| \right)^{2+\d} \right] < \infty. 
\]
Then, 
\[
 \limsup_{n\to\infty} n^{\d/2} \sup_{x \in \R} \left| \bP\left( \frac{X_n- \mu n}{\s \sqrt{n}} \leq x  \right) - \Phi(x) \right| < \infty,
\]
where $\mu$ and $\s$ are defined as in \eqref{eq:regLLN} and \eqref{eq:regCLT}, respectively. 
\end{prop}

\begin{proof}
 For notational convenience, in the proof below we will let $\bar{X}_n = X_n - n\mu$. 
The strategy of the proof of Proposition \ref{pr:BEregbP} will be to condition on the time and value of the regenerative process at the last regeneration time prior to time $n$. To this end, let $k(n) \geq 0$ be the number of regeneration times that have occured by time $n$; that is, 
$\tau_{k(n)} \leq n < \tau_{k(n)+1}$.
By decomposing according to the values of $k(n)$, $n-\tau_{k(n)}$ and $X_n-X_{\tau_{k(n)}}$, we can write 
\begin{align*}
 \bP\left( \frac{\bar{X}_n}{\s\sqrt{n}} \leq x \right) = \sum_{k=0}^n \sum_{m=0}^n \int \bP\left( \frac{\bar{X}_n}{\s\sqrt{n}} \leq x, \, k(n) = k, \, \tau_k = n-m, \, X_n - X_{\tau_k} \in du \right).
\end{align*}
Using the structure provided by the regeneration times, for any fixed $k,m$, and $u$ we can re-write the probability inside the sums and integral on the right as 
\begin{align*}
& \bP\left( \frac{\bar{X}_n}{\s\sqrt{n}} \leq x, \, k(n) = k, \, \tau_k = n-m, \, X_n - X_{\tau_k} \in du \right) \\
&\qquad = \bP\left( X_{\tau_k} - \tau_k \mu \leq x\s \sqrt{n}- u + (n-\tau_k) \mu, \, \tau_k = n-m, \, \tau_{k+1} > n, \, X_n - X_{\tau_k} \in du \right) \\
&\qquad = \bP\left( \frac{\bar{X}_{\tau_k} }{\sqrt{k}}  \leq \frac{x\s \sqrt{n}- u + m \mu}{\sqrt{k}} , \, \tau_k = n-m \right) \bP\left( \tau_1 > m, \, X_m \in du \right), 
\end{align*}
and therefore, 
\begin{align*}
 \bP\left( \frac{\bar{X}_n}{\s\sqrt{n}} \leq x \right)
&= \sum_{k=1}^n \sum_{m=0}^{\fl{\sqrt{n}}} \int_{-\sqrt{n}}^{\sqrt{n}} \bP\left( \frac{\bar{X}_{\tau_k}}{\sqrt{k}} \leq  \frac{x \s \sqrt{n} - u+m \mu}{\sqrt{k}} , \, \tau_k = n-m\right) \bP\left( \tau_1 > m, \, X_m\in du \right) \\
&\qquad + \bP\left( \{\bar{X}_n \leq x \s \sqrt{n} \} \cap \left\{ n-\tau_{k(n)} > \sqrt{n}, \text{ or } |X_n-X_{\tau_{k(n)}}| > \sqrt{n} \right\} \right).
\end{align*}
Note that in the above we could have included the terms $m > \sqrt{n}$ and $|u|> \sqrt{n}$ in the first term on the right and omitted the second term. However, the main contribution will come from $m,|u| \leq \sqrt{n}$ and thus to simplify later parts of the proof we choose to handle the cases where $n-\tau_{k(n)} > \sqrt{n}$ or $|X_n - X_{\tau_{k(n)}}|>\sqrt{n}$ separately. 
Note also that we have ommited $k=0$ from the first sum since this is included in the last term since $\tau_0 = 0$.

To use this decomposition to compare with $\Phi(x)$, note first of all that letting $\bar\tau = \bE[\tau_1]$ we can write 
\begin{align*}
\Phi(x) &= 
\frac{\Phi(x)}{\bar\tau} \sum_{m=0}^\infty \bP(\tau_1 > m) = 
\frac{\Phi(x)}{\bar\tau}\sum_{m=0}^\infty \int \bP(\tau_1 > m, \, X_m \in dy ) \\
&= \frac{\Phi(x)}{\bar\tau}\sum_{m=0}^{\fl{\sqrt{n}}} \int_{-\sqrt{n}}^{\sqrt{n}} \bP(\tau_1 > m, \, X_m \in dy )
+ \frac{\Phi(x)}{\bar\tau}\sum_{m=0}^{\fl{\sqrt{n}}} \bP(\tau_1 > m, \, |X_m| > \sqrt{n} ) \\
&\qquad + \frac{\Phi(x)}{\bar\tau} \sum_{m>\sqrt{n}} \bP(\tau_1 > m). 
\end{align*}
Therefore, we can conclude that 
\begin{align}
&  \bP\left( \frac{\bar{X}_n}{\s\sqrt{n}} \leq x \right) - \Phi(x)  \nonumber \\
&\quad = \sum_{m=0}^{\fl{\sqrt{n}}} \int_{-\sqrt{n}}^{\sqrt{n}} \left\{ \sum_{k=1}^n \bP\left( \tfrac{\bar{X}_{\tau_k}}{\sqrt{k}} \leq  \tfrac{x \s \sqrt{n} - y+m \mu}{\sqrt{k}} , \, \tau_k = n-m\right) - \frac{\Phi(x)}{\bar\tau}  \right\} \bP\left( \tau_1 > m, \, X_m\in dy \right)  \label{mainterm} \\
&\quad\qquad + \bP\left( \{\bar{X}_n \leq x \s \sqrt{n} z\} \cap \left\{ n-\tau_{k(n)} > \sqrt{n}, \text{ or } |X_n-X_{\tau_{k(n)}}| > \sqrt{n} \right\} \right) \label{regerr} \\
&\quad\qquad - \frac{\Phi(x)}{\bar\tau}\sum_{m=0}^{\fl{\sqrt{n}}} \bP(\tau_1 > m, \, |X_m| > \sqrt{n} ) - \frac{\Phi(x)}{\bar\tau} \sum_{m>\sqrt{n}} \bP(\tau_1 > m). \label{Phierr}
\end{align}
To control the terms in \eqref{regerr}, note that the moment assumptions in the statement of the theorem imply that 
\begin{align*}
\eqref{regerr} 
&\leq n \bP(\tau_1 > \sqrt{n}) + n \bP\left( \sum_{i=1}^{\tau_1} |\xi_i| > \sqrt{n} \right) 
= \bigo( n^{-\d/2} ).
\end{align*}
Similarly, the terms in \eqref{Phierr} can be bounded by 
\begin{align*}
\eqref{Phierr} &\leq \frac{1+\sqrt{n}}{\bar\tau}  \bP\left( \sum_{i=1}^{\tau_1} |\xi_i| > \sqrt{n} \right) + \frac{\bE[\tau_1^{2+\d}]}{\bar\tau} \sum_{m>\sqrt{n}} (m+1)^{-2-\d} 
= \bigo( n^{-(1+\d)/2} ). 
\end{align*}
Therefore, it remains only to show that the term in \eqref{mainterm} is also $\bigo(n^{-\d/2})$, uniformly in $x$. 
To this end, let $\psi_A(x,y)$ be defined as in \eqref{psidef}, where 
\[
 A = 
\begin{pmatrix}
 \bE[(X_{\tau_1} - \tau_1 \mu)^2 ] & \bE[(X_{\tau_1} - \tau_1 \mu)(\tau_1 - \bar\tau) ]  \\
 \bE[(X_{\tau_1} - \tau_1 \mu)(\tau_1 - \bar\tau) ] & \bE[ (\tau_1 - \bar\tau)^2 ] \\
\end{pmatrix}
\]
is the covariance matrix of $(X_{\tau_1}-\tau_1 \mu, \tau_1)$ under the measure $\bP$. 
For convenience of notation, let 
\begin{equation}\label{alphadef}
 \a^2 = \bE\left[(X_{\tau_1} - \tau_1 \mu)^2 \right]
\end{equation}
be the top left entry of the covariance matrix $A$. If $\mathcal{N} = (\mathcal{N}_1,\mathcal{N}_2)$ is a centered Gaussian with covariance matrix $A$, then it follows that $\frac{\mathcal{N}_1}{\a}$ is a standard Normal random variable and thus
\[
 \int_\R \psi_A(\a x, y) \, dy = P( \mathcal{N}_1 \leq \a x ) = \Phi(x). 
\]
Using this notation, the necessary bounds on \eqref{mainterm} which complete the proof of Proposition \ref{pr:BEregbP} are obtained by a series of approximations given by the following three lemmas. Note that in these lemmas and below we will use the following notation.
\begin{equation}\label{ykdef}
 y_{k,n,m} = \frac{n-m-k\bar\tau}{\sqrt{k}}. 
\end{equation}

\begin{lem}\label{lem:part1}
There exists a constant $C<\infty$ such that for $n$ large enough, 
 \begin{align*}
 & \sum_{m=0}^{\fl{\sqrt{n}}} \int_{-\sqrt{n}}^{\sqrt{n}} \sum_{k=1}^n  \biggl| \bP\left( \tfrac{\bar{X}_{\tau_k}}{\sqrt{k}} \leq  \tfrac{x \s \sqrt{n} - u+m \mu}{\sqrt{k}} , \, \tau_k = n-m\right) 
 \\
&\hspace{1.5in} - \frac{1}{\sqrt{k}} \psi_A\left( \tfrac{x \s \sqrt{n}-u+m \mu}{\sqrt{k}}, y_{k,n,m} \right) \biggr| \bP\left( \tau_1 > m, \, X_m\in du \right)
\leq \frac{C}{n^{\d/2}},
 \end{align*}
for all $x \in \R$.
\end{lem}

\begin{lem}\label{lem:part2}
There exists a constant $C<\infty$ such that for $n$ large enough, 
\begin{align*}
& \sum_{m=0}^{\fl{\sqrt{n}}} \sum_{k=1}^n \Biggl| \int_{-\sqrt{n}}^{\sqrt{n}} \frac{1}{\sqrt{k}} \psi_A\left(\tfrac{ x\s\sqrt{n}-u+m\mu}{\sqrt{k}}, y_{k,n,m} \right) \bP\left( \tau_1 > m, X_m\in du \right) \\
&\hspace{1.5in}  - \frac{1}{\sqrt{k}} \psi_A\left(\a x , y_{k,n,m} \right) \bP\left( \tau_1 > m, |X_m|\leq \sqrt{n} \right) \Biggr| 
\leq \frac{C}{\sqrt{n}},
\end{align*}
for all $x \in \R$. 
\end{lem}

\begin{lem}\label{lem:part3}
There exists a constant $C<\infty$ such that for $n$ large enough, 
 \[
  \sum_{m=0}^{\fl{\sqrt{n}}} \left| \sum_{k=1}^n \frac{1}{\sqrt{k}} \psi_A\left(\a x, y_{k,n,m} \right)  - \frac{1}{\bar\tau} \int_\R \psi_A(\a x,y) \, dy \right| \bP\left( \tau_1 > m, |X_m|\leq \sqrt{n} \right) \leq \frac{C}{\sqrt{n}},
 \]
for all $x \in \R$. 
\end{lem}



\begin{proof}[Proof of Lemma \ref{lem:part1}]
It follows from Theorem \ref{thm:semi-local-LLT} that the sum in the statement of the Lemma is bounded by 
\begin{equation}\label{applysemilocal}
 \sum_{m=0}^{\fl{\sqrt{n}}} \int_{-\sqrt{n}}^{\sqrt{n}} \sum_{k=1}^n 
\frac{C}{k^{(1+\d)/2}} \left(1+\frac{(n-m-k\bar\tau)^2}{k} \right)^{-1} \bP(\tau_1 > m, \, X_m \in du ).   
\end{equation}
(A direct application of Theorem \ref{thm:semi-local-LLT} requires that the random variable $\tau_1$ has span 1 under the law $\bP$, but clearly Theorem \ref{thm:semi-local-LLT} can be generalized to any lattice random variable $\bvar$.)
Note that 
for $m\leq \sqrt{n}$ we can bound
\begin{align*}
 \frac{1}{k^{(1+\d)/2}} \left(1+\frac{(n-m-k\bar\tau)^2}{k} \right)^{-1}
&\leq
\begin{cases}
 \frac{k^{(1-\d)/2}}{(n-\sqrt{n}-k\bar\tau)^2} & \text{if } 1 \leq k < \frac{n-2 \sqrt{n}}{\bar\tau} \\
 \frac{1}{k^{(1+\d)/2}} & \text{if } |n-k\bar\tau| \leq 2\sqrt{n} \\
 \frac{k^{(1-\d)/2}}{(n-k\bar\tau)^2} & \text{if } \frac{n+2 \sqrt{n}}{\bar\tau} < k \leq n,
\end{cases}
\end{align*}
and from this it follows easily (using integrals to bound the appropriate sums) that 
\[
 \sum_{k=1}^n \frac{1}{k^{(1+\d)/2}} \left(1+\frac{(n-m-k\bar\tau)^2}{k} \right)^{-1} \leq  \frac{C}{n^{\d/2}}, 
\]
for some $C<\infty$. 
Therefore, we obtain that 
\begin{align*}
\eqref{applysemilocal}
&\leq \sum_{m=0}^{\fl{\sqrt{n}}} \int_{-\sqrt{n}}^{\sqrt{n}} \frac{C}{n^{\d/2}} \bP(\tau_1 > m, \, X_m \in dy ) 
\leq \frac{C}{n^{\d/2}} \sum_{m=0}^{\fl{\sqrt{n}}} \bP(\tau_1 > m ) 
\leq \frac{C \bE[\tau_1]}{n^{\d/2}}.
\end{align*}
\end{proof}

Before giving the proofs of Lemmas \ref{lem:part2} and \ref{lem:part3}, we first state the following facts which were used in the proofs of the corresponding statements in \cite{bo80}. 
\begin{lem}
 Let $y_{k,n,m} = \frac{n-m-k\bar\tau}{\sqrt{k}}$. For any constant $c>0$, there exists a constant $C<\infty$ depending only on $c$ and $\bar\tau$ such that 
\begin{equation}\label{sumey2a}
 \sup_{m\leq \sqrt{n}} \sum_{k=1}^n \frac{1}{k} e^{-c y_{k,n,m}^2} \leq \frac{C}{\sqrt{n}}
\end{equation}
and 
\begin{equation}\label{sumey2b}
 \sup_{m\leq \sqrt{n}} \sum_{k=1}^n \frac{1}{\sqrt{k}}\left| \sqrt{\frac{n}{k \bar\tau}} - 1 \right| e^{-c y_{k,n,m}^2} \leq \frac{C}{\sqrt{n}}.
\end{equation}
\end{lem}

\begin{rem}
 We refer the reader to pages 69-70 in \cite{bo80} for the proofs of \eqref{sumey2a} and \eqref{sumey2b}. 
\end{rem}

\begin{proof}[Proof of Lemma \ref{lem:part2}]
Let $I(k,m,x,y)$ denote the interval between $\a x$ and $\frac{ x\s\sqrt{n}-y+m\mu}{\sqrt{k}}$, and recall that $\gamma_A(x,y)$ is the p.d.f.\ of a centered two dimensional Gaussian with covariance matrix $A$. 
Then, 
\begin{align*}
& \left|  \psi_A\left(\tfrac{ x\s\sqrt{n}-y+m\mu}{\sqrt{k}}, y_{k,n,m} \right) -  \psi_A\left( \a x, y_{k,n,m} \right) \right| \\
&\qquad = \left| \int_{I(k,m,x,y)} \gamma_A\left( z, y_{k,n,m} \right) \, dz \right| \\
&\qquad \leq \left( |x|\left| \frac{\s \sqrt{n}}{\sqrt{k}} - \a \right| + \frac{|y-m \mu|}{\sqrt{k}} \right) \sup_{z \in I(k,m,x,y)} \gamma_A\left( z, y_{k,n,m} \right)
\end{align*}
Next, note that there exist constants $c_1,c_2>0$ depending only on the entries of the covariance matrix $A$ such that 
\begin{equation}\label{phibounds}
 \gamma_A(x,y) \leq c_1 e^{-c_2(x^2+y^2)}. 
\end{equation}
Therefore, 
\begin{align}
& \text{(Left side of Lemma \ref{lem:part2}) } \nonumber \\
&\leq 
\sum_{m=0}^{\sqrt{n}} \sum_{k=1}^n \frac{c_1}{k}  e^{-c_2 y_{k,n,m}^2} \int_{-\sqrt{n}}^{\sqrt{n}} |y-m\mu|  \bP(\tau_1>m, \, X_m \in dy) \label{p2dec1} \\
& \quad + \sum_{m=0}^{\sqrt{n}} \sum_{k=1}^n \frac{c_1|x|}{\sqrt{k}}\left| \frac{\s \sqrt{n}}{\sqrt{k}} - \a \right|   e^{-c_2 y_{k,n,m}^2 } \int_{-\sqrt{n}}^{\sqrt{n}}  \sup_{z \in I(k,m,x,y)} e^{-c_2 z^2} \,  \bP(\tau_1>m, \, X_m \in dy) \label{p2dec2}
\end{align}

To control \eqref{p2dec1}, note that the integral inside the sums is zero if $m=0$ whereas for $m\geq 1$ we have  
\begin{align*}
 \int_{-\sqrt{n}}^{\sqrt{n}} |y-m\mu|  \bP(\tau_1>m, \, X_m \in dy) 
&= \bE\left[ |X_m - m\mu| \ind{\tau_1>m} \right] \\
&\leq \bE\left[ \sum_{i=1}^{\tau_1} |\xi_i - \mu| \ind{\tau_1 > m} \right] 
\leq C \bP(\tau_1 > m)^{\frac{1+\d}{2+\d}} 
\leq \frac{C' }{m^{1+\d}}, 
\end{align*}
using the moment assumptions in the statement of Proposition \ref{pr:BEregbP} together with H\"older's inequality and Chebychev's inequality in the last two inequalities, respectively. 
From this and \eqref{sumey2a}, we obtain that  
\begin{align*}
\eqref{p2dec1} \leq \sum_{m=1}^{\sqrt{n}} \frac{C}{m^{1+\d}} \sum_{k=1}^n \frac{1}{k}  e^{-c_2 y_{k,n,m}^2 } 
\leq \frac{C'}{\sqrt{n}} \sum_{m=1}^{\sqrt{n}} \frac{1}{ m^{1+\d}} = \frac{C''}{\sqrt{n}}. 
\end{align*}

To control \eqref{p2dec2}, we claim that 
\begin{equation}\label{supez}
 \sup_{z \in I(k,m,x,y)} |x| e^{-c_2 z^2} \leq C. 
\end{equation}
To see this, first note that since $m,|y|\leq \sqrt{n}$ and $k\leq \sqrt{n}$ it follows that 
\[
 \left| \frac{ x\s\sqrt{n}-y+m\mu}{\sqrt{k}} \right| 
\geq \frac{|x| \s \sqrt{n} - (1+\mu)\sqrt{n}}{\sqrt{k}} \geq |x|\s - (1+\mu).  
\]
If $|x| >  \frac{2(1+\mu)}{\s}$ then the right side can be bounded below by $|x|\s/2$ and thus $|z| > \min\{\a, \s/2\} |x|$ for all $z \in I(k,m,x,y)$. 
Therefore, 
\[
 \sup_{z \in I(k,m,x,y)} e^{-c_2 z^2}
\leq 
\begin{cases}
 1 & \text{if } |x| \leq \frac{2(1+\mu)}{\s} \\
e^{-c_2 \min\{ \a, \s/2\} |x|^2 } & \text{if } |x| > \frac{2(1+\mu)}{\s},
\end{cases}
\]
and from this the claim in \eqref{supez} follows. 
Using \eqref{supez} and then \eqref{sumey2b} we then have that 
\begin{align*}
 \eqref{p2dec2}
\leq C \sum_{m=0}^{\sqrt{n}} \left( \sum_{k=1}^n \frac{1}{\sqrt{k}}\left| \frac{\s \sqrt{n}}{\sqrt{k}} - \a \right|   e^{-c_2 y_{k,n,m}^2 } \right) \bP(\tau_1>m )
\leq  \frac{C'}{\sqrt{n}} \sum_{m=0}^{\sqrt{n}} \bP(\tau_1>m ) \leq \frac{C' \bar\tau}{\sqrt{n}}. 
\end{align*}
(Note that in the application of \eqref{sumey2b} we are using that 
$\s^2 \bar\tau = \a^2$
which follows from the definitions of $\s^2$ and $\a^2$ in \eqref{eq:regCLT} and \eqref{alphadef}, respectively.)
\end{proof}


\begin{proof}[Proof of Lemma \ref{lem:part3}]
In the proof of this Lemma, to make the notation less burdensome, in a slight abuse of notation we will write $y_k$ for $y_{k,n,m}$ as defined in \eqref{ykdef}. 
To begin, note for any fixed $n\geq m$ that $y_1 > y_2 >\cdots > y_n$.
Since for $n$ large enough and $m\leq \sqrt{n}$ we have $y_1 = n-m-\bar\tau \geq n/2$, if $\mathcal{N} = (\mathcal{N}_1,\mathcal{N}_2)$ is a centered Gaussian random variable with Covariance matrix $A$, then 
\begin{align*}
 \sum_{m=0}^{\fl{\sqrt{n}}} \frac{1}{\bar\tau} \left( \int_{y_1}^\infty \psi_A(\a x, y) \, dy \right) \bP(\tau_1 > m, |X_m|\leq \sqrt{n}) 
&\leq \sum_{m=0}^{\fl{\sqrt{n}}} \frac{1}{\bar\tau} P(\mathcal{N}_2 \geq n/2) \bP(\tau_1 > m) \\
&\leq P(\mathcal{N}_2 \geq n/2) = o(n^{-1/2}). 
\end{align*}
Similarly, since $y_n \leq -(\bar\tau-1) \sqrt{n}$ and $\bar\tau = \bE[\tau_1] > 1$ (otherwise the regenerative process is simply an i.i.d.\ sequence), then 
\begin{align*}
\sum_{m=0}^{\fl{\sqrt{n}}} \frac{1}{\bar\tau} \left( \int_{-\infty}^{y_n} \psi_A(\a x, y) \, dy \right) \bP(\tau_1 > m, |X_m|\leq \sqrt{n}) 
\leq P(\mathcal{N}_2 \leq -(\bar\tau-1)\sqrt{n} ) = o(n^{-1/2}). 
\end{align*}
Therefore, to finish the proof of Lemma \ref{lem:part3} it is enough to show that 
\begin{equation}\label{intparts}
 \sum_{m=0}^{\fl{\sqrt{n}}} \sum_{k=1}^{n-1} \left| \frac{1}{\sqrt{k}} \psi_A(\a x, y_k) - \frac{1}{\bar\tau} \int_{y_{k+1}}^{y_k} \psi_A(\a x,y) \, dy \right| \bP(\tau_1 > m, |X_m|\leq \sqrt{n}) = \bigo(n^{-1/2}). 
\end{equation}

To prove \eqref{intparts}, first note that 
\begin{align}
& \left| \frac{1}{\sqrt{k}} \psi_A(\a x, y_k) - \frac{1}{\bar\tau} \int_{y_{k+1}}^{y_k} \psi_A(\a x,y) \, dy \right| \nonumber \\
&\quad \leq \left| \frac{1}{\sqrt{k}} - \frac{y_k-y_{k+1}}{\bar\tau} \right|\psi_A(\a x, y_k) 
 + \frac{1}{\bar\tau}  \int_{y_{k+1}}^{y_k} \left| \psi_A(\a x,y) - \psi_A(\a x,y_k) \right| \, dy. \label{psicompare}
\end{align}
To control the first term in \eqref{psicompare}, the definition of $y_k$ implies that 
$y_k = \frac{\bar\tau}{\sqrt{k}} + y_{k+1} \sqrt{\frac{k+1}{k}}$,
or equivalently, 
\begin{equation}\label{ykdiff}
 y_k- y_{k+1} = \frac{\bar\tau}{\sqrt{k}} + y_{k+1} \left(\sqrt{1+\frac{1}{k}} - 1\right).
\end{equation}
Since $\sqrt{1+\frac{1}{k}}-1 \leq \frac{1}{2k}$ we can conclude from this that  
\begin{equation}\label{ydiffpart}
 \left| \frac{1}{\sqrt{k}} - \frac{y_k-y_{k+1}}{\bar\tau} \right|\psi_A(\a x, y_k) 
= \frac{y_{k+1}}{\bar\tau} \left| \sqrt{1+\frac{1}{k}} - 1 \right| \psi_A(\a x, y_k) 
\leq \frac{C}{k} y_{k+1} e^{-c y_k^2}, 
\end{equation}
where in the last inequality we used that the bounds on $\gamma_A$ in \eqref{phibounds} imply that $\psi_A(z,y) \leq C e^{-cy^2}$.

To control the second term in \eqref{psicompare}, note that for $y \in [y_{k+1},y_k]$, 
\begin{equation}\label{psidiff}
 \left| \psi_A(\a x,y) - \psi_A(\a x,y_k) \right| \leq C |y_k-y_{k+1}| \sup_{y \in [y_{k+1},y_k]} e^{-c y^2}
\end{equation}
To further simplify the supremum on the right, note that for any $y \in [y_{k+1},y_k]$ 
\begin{align*}
 y_{k+1}^2 \leq 2y^2 + 2(y - y_{k+1})^2 
\leq 2y^2 + 2(y_k-y_{k+1})^2 \leq 2 y^2 + \frac{4 \bar\tau^2}{k} + \frac{y_{k+1}^2}{k^2}, 
\end{align*}
where we used \eqref{ykdiff} in the last inequality. 
For $k\geq 2$ this implies that $\inf_{y \in [y_{k+1},y_k]} y^2 \geq \frac{3}{8} y_{k+1}^2 - \bar\tau^2 $, and this is also trivially true for $k=1$ since $0<y_2<y_1$ so that we can conclude 
\begin{equation}\label{ey2}
 \sup_{y \in [y_{k+1},y_k]} e^{-c y^2} \leq C e^{-\frac{3 c}{8} y_{k+1}^2}. 
\end{equation}
Using \eqref{psidiff}, \eqref{ey2} and then \eqref{ykdiff} we can bound the second term in \eqref{psicompare} by 
\begin{align*}
 \frac{1}{\bar\tau} \int_{y_{k+1}}^{y_k} \left| \psi_A(\a x,y) - \psi_A(\a x,y_k) \right| \, dy 
&\leq C |y_k-y_{k+1}|^2 e^{-c y_{k+1}^2} \\
&\leq C' \left( \frac{1}{k} + \frac{y_{k+1}^2}{k^2} \right) e^{-c y_{k+1}^2} 
\leq \frac{C''}{k} e^{-c' y_{k+1}^2 }. 
\end{align*}
Combining this with \eqref{ydiffpart} and \eqref{psicompare} we obtain that 
\[
 \left| \frac{1}{\sqrt{k}} \psi_A(\a x, y_k) - \frac{1}{\bar\tau} \int_{y_{k+1}}^{y_k} \psi_A(\a x,y) \, dy \right| \leq \frac{C}{k} y_{k+1} e^{-c y_k^2} + \frac{C}{k} e^{-c y_{k+1}^2} \leq \frac{C'}{k} e^{-c' y_{k+1}^2} 
\]
and thus,
\begin{align*}
& \sum_{m=0}^{\fl{\sqrt{n}}} \sum_{k=1}^{n-1} \left| \frac{1}{\sqrt{k}} \psi_A(\a x, y_k) - \frac{1}{\bar\tau} \int_{y_{k+1}}^{y_k} \psi_A(\a x,y) \, dy \right| \bP(\tau_1 > m, |X_m|\leq \sqrt{n}) \\
&\leq \sum_{m=0}^{\fl{\sqrt{n}}} \sum_{k=1}^n \frac{C'}{k} e^{-c' y_{k+1}^2} \bP(\tau_1 > m, |X_m|\leq \sqrt{n}) \\
&\leq \frac{C''}{\sqrt{n}} \sum_{m=0}^{\fl{\sqrt{n}}} \bP(\tau_1 > m) \leq \frac{C''\bar\tau}{\sqrt{n}},
\end{align*}
where we used \eqref{sumey2a} in the second to last inequality. 
\end{proof}

\end{proof}

\subsection{Accounting for the first regeneration interval}

In this subsection, we will show how to account for the difference of the first regeneration interval to improve Proposition \ref{pr:BEregbP} to a proof of Theorem \ref{th:BEregCLT}. 

\begin{proof}[Proof of Theorem \ref{th:BEregCLT}]
By conditioning on the values of $\tau_1$ and $X_{\tau_1}$ we obtain that 
 \begin{align*}
  \P\left( \frac{X_n - n\mu}{\s \sqrt{n}} \leq t \right) 
&= \sum_{m=1}^{\fl{\sqrt{n}}} \int_{-\sqrt{n}}^{\sqrt{n}} \P(X_{\tau_1} \in dz, \, \tau_1 = m) \bP( X_{n-m} - (n-m)\mu \leq \s t \sqrt{n} -z+m\mu) \\
&\qquad + \P\left( \frac{X_n - n\mu}{\s\sqrt{n}} \leq t, \text{ and } \max\{ |X_{\tau_1}|, \tau_1 \} > \sqrt{n} \right)
 \end{align*}
Since 
\[
 \Phi(t)
 = \sum_{m=1}^{\fl{\sqrt{n}}} \int \P(X_{\tau_1} \in dz, \, \tau_1 = m)\Phi(t) + \P\left(\max\{ |X_{\tau_1}|, \tau_1 \} > \sqrt{n}\right)\Phi(t),
\]
by comparing like terms we obtain 
\begin{align*}
& \left| \P\left( \frac{X_n - n\mu}{\s \sqrt{n}} \leq t \right) - \Phi(t) \right|\\
&\leq \sum_{m=1}^{\fl{\sqrt{n}}} \int_{-\sqrt{n}}^{\sqrt{n}} \P(X_{\tau_1} \in dz, \, \tau_1 = m)
\left| \bP( X_{n-m} - (n-m)\mu \leq \s t \sqrt{n} -z+m\mu) - \Phi(t) \right| \\
&\qquad + 2 \P(\tau_1 > \sqrt{n}) + 2 \P(|X_{\tau_1}| > \sqrt{n}) \\
&\leq \sum_{m=1}^{\fl{\sqrt{n}}} \int_{-\sqrt{n}}^{\sqrt{n}} \P(X_{\tau_1} \in dz, \, \tau_1 = m)
\left\{ \sup_s \left| \bP\left( \frac{\bar{X}_{n-m}}{\s\sqrt{n-m}} \leq s \right) - \Phi(s)   \right| \right\} \\
&\qquad + \sum_{m=1}^{\fl{\sqrt{n}}} \int_{-\sqrt{n}}^{\sqrt{n}} \P(X_{\tau_1} \in dz, \, \tau_1 = m)
\left| \Phi\left(  t \sqrt{\tfrac{n}{n-m}} -\tfrac{z-m\mu}{\s\sqrt{n-m}} \right) - \Phi(t) \right| \\
&\qquad + \frac{2( \E[\tau_1^\d] + \E[|X_{\tau_1}|^\d] )}{n^{\d/2}}.
\end{align*}
For $n$ large enough and $m\leq \sqrt{n}$ we have from Proposition \ref{pr:BEregbP} that 
the supremum in braces on the right is bounded by $C/\sqrt{n}$ for $n$ large enough. 
Therefore, we need only to show that 
\begin{equation}\label{Phidiff}
\limsup_{n\to\infty} n^{\d/2} \sup_{t\in \R} \sum_{m=1}^{\fl{\sqrt{n}}} \int_{-\sqrt{n}}^{\sqrt{n}} \P(X_{\tau_1} \in dz, \, \tau_1 = m)
\left| \Phi\left(  t \sqrt{\tfrac{n}{n-m}} -\tfrac{z-m\mu}{\s\sqrt{n-m}} \right) - \Phi(t) \right| < \infty. 
\end{equation}

To prove \eqref{Phidiff}, it is easy to show (see \cite[Section V.3, equations (3.3),(3.4)]{petrovRW}) that for any $a>1$ and $b,t \in \R$ that 
\[
|\Phi(at+b) - \Phi(t)| \leq |\Phi(at+b)-\Phi(at)| + |\Phi(at)-\Phi(t)| \leq \frac{1}{\sqrt{2\pi}} |b| + \frac{1}{\sqrt{2\pi e}}(a-1). 
\]
Therefore, 
\begin{align}
 & \sup_{t\in \R} \sum_{m=1}^{\fl{\sqrt{n}}} \int_{-\sqrt{n}}^{\sqrt{n}} \P(X_{\tau_1} \in dz, \, \tau_1 = m)
\left| \Phi\left(  t \sqrt{\tfrac{n}{n-m}} -\tfrac{z-m\mu}{\s\sqrt{n-m}} \right) - \Phi(t) \right| \nonumber \\
&\leq C \sum_{m=1}^{\fl{\sqrt{n}}} \int_{-\sqrt{n}}^{\sqrt{n}} \P(X_{\tau_1} \in dz, \, \tau_1 = m) \left \{ \tfrac{|z-m\mu|}{\sqrt{n-m}} +  \sqrt{\tfrac{n}{n-m}} - 1 \right\} \nonumber \\
&\leq \frac{C}{\sqrt{n-\sqrt{n}}} \E\left[|X_{\tau_1} - \mu \tau_1| \ind{|X_{\tau_1}|\leq \sqrt{n}, \tau_1 \leq \sqrt{n}} \right] + C \left( \sqrt{\tfrac{n}{n-\sqrt{n}}} - 1 \right). \label{Phidiff2}
\end{align}
Finally, using the moment assumptions regarding the first regeneration time we have that 
\begin{equation}\label{Xtausmall}
 \E\left[|X_{\tau_1} - \mu \tau_1| \ind{|X_{\tau_1}|\leq \sqrt{n}, \tau_1 \leq \sqrt{n}} \right]
\leq (1+\mu)^{1-\d} n^{(1-\d)/2} \E\left[|X_{\tau_1} - \mu \tau_1|^\d \right]. 
\end{equation}
Applying \eqref{Xtausmall} to \eqref{Phidiff2},
we see that \eqref{Phidiff} follows easily.
\end{proof}

\section{Discussions: rates of convergence of quenched and annealed CLT of RWRE}\label{sec:disc}

The results in Section~\ref{sec:RWRE} give rates of convergence for annealed CLTs of RWRE. 
However, under certain assumptions it is known that CLTs hold under the quenched measures as well. 
Below we will review
some recent results on the corresponding quenched rates of convergence for one-dimensional RWRE. 
We will then close the paper with a few related open questions.

\subsection{One-dimensional quenched CLTs}
Recall from \eqref{1daCLT} that one-dimensional RWREs with parameter $\k>2$ have annealed CLTs for both the position of the walk and the hitting times of the walk. 
It is known that the position and hitting times of the walk also have Gaussian limiting distributions under the quenched measure $P_\w$ (for $P$-a.e.\ environment $\w$), but that the centering and scaling needs to be somewhat different than  in the annealed CLTs \cite{aRWRE,gQCLT,pThesis}.
In particular, 
\begin{equation}\label{qcltTn}
 \lim_{n\to\infty} \sup_x\left| P_\w\left( \frac{T_n - E_\w[T_n]}{\s_1 \sqrt{n}} \leq x \right) - \Phi(x)\right| = 0, \quad P\text{-a.s.}, \quad \text{where } \s_1^2 = E[\Var_\w(T_1)], 
\end{equation}
and 
\[
 \lim_{n\to\infty} \sup_x \left| P_\w\left( \frac{X_n - n\vp + Z_n(\w)}{\vp^{3/2} \s_1 \sqrt{n}} \leq x \right) - \Phi(x)\right|= 0, \quad P\text{-a.s.}, 
\]
where $Z_n(\w) = \vp\left( E_\w[T_{\fl{n\vp}}] - \E[T_{\fl{n\vp}}] \right).$

Recent results of Ahn and Peterson \cite{apQCLTrates} gave upper bounds for the rates of convergence of these quenched CLTs. 
While the quenched CLT for the hitting times stated in \eqref{qcltTn} had a quenched centering and a deterministic scaling,
the results in \cite{apQCLTrates} show that improved rates of convergence can be obtained for the hitting times by using a quenched scaling as well. 
\begin{thm}[Ahn and Peterson \cite{apQCLTrates}]\label{th:QCLTratesT}
Let 
\[
 F_{n,\w}(x) = P_\w\left( \frac{T_n-E_\w[T_n]}{\s_1 \sqrt{n}} \leq x \right) \quad\text{and} \quad \overline{F}_{n,\w}(x) = P_\w\left(\frac{T_n - E_\w[T_n]}{\sqrt{\Var_\w(T_n)}} \leq x \right)
\]
be the centered quenched distribution functions of $T_n$ with deterministic and quenched scalings, respectively. 
\begin{enumerate}
 \item Rates of convergence with deterministic scaling:
\begin{enumerate}
 \item If $\k > 4$, then for any $\e>0$, 
\[
 \lim_{n\to\infty} n^{\frac{1}{2}-\e} \| F_{n,\w} - \Phi \|_\infty = 0, \quad P\text{-a.s.}
\]
 \item If $\k \in (2,4]$, then for any $\e > 0$, 
\[
 \lim_{n\to\infty} n^{1-\frac{2}{\k}-\e} \| F_{n,\w} - \Phi \|_\infty = 0, \quad P\text{-a.s.}
\]
\end{enumerate}

 \item Rates of convergence with quenched scaling. 
 \begin{enumerate}
  \item If $\k > 3$, then there exists a constant $C<\infty$ such that 
 \[
  \limsup_{n\to\infty} \sqrt{n} \|\overline{F}_{n,\w} - \Phi \|_\infty \leq C, \quad P\text{-a.s.}
 \]
  \item If $\k \in (2,3]$ then for any $\e>0$, 
  \[
   \lim_{n\to\infty} n^{\frac{3}{2} - \frac{3}{\k} - \e} \| \overline{F}_{n,\w}-\Phi\|_\infty = 0, \quad P\text{-a.s.}
  \]
 \end{enumerate}
\end{enumerate}
\end{thm}
The corresponding results for the quenched CLT of the position of the walk are somewhat weaker but don't require a quenched scaling.  
\begin{thm}[Ahn and Peterson \cite{apQCLTrates}]\label{th:QCLTratesX}
 Let $G_{n,\w}(x) = P_\w\left( \frac{X_n - n\vp + Z_n(\w)}{\vp^{3/2}\s_1 \sqrt{n}} \leq x \right)$ be the rescaled quenched distribution function of $X_n$. 
 If $\k >2$, then for any $\e>0$
  \[
   \limsup_{n\to\infty} n^{\frac{1}{4} - \frac{1}{2\k} - \e} \|G_{n,\w} - \Phi\|_\infty = 0, \quad P\text{-a.s.}
  \]
 Moreover, by relaxing the convergence to that of in probability one obtains the following faster rates of convergence. 
 \begin{enumerate}
  \item If $\k \in (2,\frac{12}{5})$, then for any $\e>0$,
  \begin{equation}\label{QCLTrates1}
   \limsup_{n\to\infty} n^{\frac{3}{2} - \frac{3}{\k} - \e} \|G_{n,\w} - \Phi\|_\infty = 0, \quad \text{in $P$-probability.}
  \end{equation}
  \item If $\k > \frac{12}{5}$ then for any $\e>0$, 
 \[
   \limsup_{n\to\infty} n^{\frac{1}{4} - \e} \|G_{n,\w} - \Phi\|_\infty = 0, \quad \text{in $P$-probability.}
  \]
 \end{enumerate}
\end{thm}

\subsection{Remaining questions for quenched and annealed rates of convergence}

\begin{enumerate}
 \item The rates of convergence of the annealed CLTs in  Corollary \ref{cor:1dratesX} are clearly optimal when $\kappa > 3$. However, since $\frac{3}{2}-\frac{3}{\k} > \frac{\k}{2}-1$ when $\k \in (2,3)$, the results in Theorems \ref{th:QCLTratesT} and \ref{th:QCLTratesX} prompt one to consider whether one can obtained better rates of convergence for the annealed CLTs by using quenched centerings and/or scalings.
 In particular, is it true that 
 \[
  \sup_{x \in \R} \left| \P\left( \frac{T_n - E_\w[T_n]}{\sqrt{Var_\w(T_n)}} \leq x \right) - \Phi(x) \right| = o(n^{-\frac{\k}{2} + 1}), 
 \]
 when $\k \in (2,3)$? 
 \item For multidimensional RWRE, under strong enough moment conditions on the regeneration times it is known that a quenched CLT holds \cite{bzQCLT,rsQCLT}. Moreover, in contrast to the one-dimensional case, the quenched CLT holds with the same (deterministic) centering and scaling as the annealed CLT. 
 Can one prove rates of convergence for the quenched CLT in these cases? Also, can the rate of convergence be improved by instead using a quenched centering and/or scaling instead of the deterministic one? Answering these questions will likely require techniques very different from this paper since the intervals of the walk between regeneration times are no longer i.i.d.\ under the quenched measure. 
 \item There are certain multidimensional RWRE which are not directionally transient but for which a CLT holds; for instance RWRE in balanced random environments \cite{lBCLT,gzBQIP,bdBQIP} or environments in which certain projections of the walk are a simple symmetric random walk \cite{bszCPDRW}. Since these walks are not directionally transient, the regeneration times do not even exist. Can one use other techniques to obtain rates of convergence for the quenched or annealed CLTs of these RWRE?   
\end{enumerate}

\bibliographystyle{alpha}
\bibliography{RWREref}

\end{document}